\documentclass[11pt]{amsart}
\usepackage{geometry}                
\usepackage{graphicx}
\usepackage{amssymb}
\usepackage{epstopdf}
\newtheorem{theorem}{Theorem}[section]
\newtheorem{lemma}{Lemma}[section]
\newtheorem{corollary}{Corollary}[section]

\newtheorem{conjecture}{Conjecture}[section]
\newtheorem{definition}{Definition}[section]

\newtheorem{question}{Question}[section]
\numberwithin{equation}{section}

\def\Z{\mathbb Z}

\def\R{\mathbb R}

\def\H{\mathbb H}

\def\d{\partial}
\def\a{\alpha}
\def\b{\beta}
\def\g{\gamma}
\def\s{\sigma}
\def\e{\epsilon}
\def\D{\Delta}

\title{Causal Holography in Application to the Inverse Scattering Problems}
\author{Gabriel Katz}

\address{MIT, Department of Mathematics, 77 Massachusetts Ave., Cambridge, MA 02139, U.S.A.}
\email{gabkatz@gmail.com}

\begin{document}
\maketitle

\begin{abstract} 
For a given smooth compact manifold $M$, we introduce an open class $\mathcal G(M)$ of Riemannian metrics, which we call \emph{metrics of the gradient type}. For such metrics $g$, the geodesic flow $v^g$ on the spherical tangent bundle $SM \to M$ admits a Lyapunov function (so the $v^g$-flow is traversing). It turns out, that metrics of the gradient type are exactly the non-trapping metrics. 

For every $g \in \mathcal G(M)$, the geodesic scattering along the boundary $\d M$ can be expressed in terms of the \emph{scattering map} $C_{v^g}: \d_1^+(SM) \to \d_1^-(SM)$. It acts from a domain $\d_1^+(SM)$ in the boundary $\d(SM)$ to the complementary domain $\d_1^-(SM)$, both domains being diffeomorphic.  We prove that, for a \emph{boundary generic} metric $g \in \mathcal G(M)$, the map $C_{v^g}$ allows for a reconstruction of $SM$ and of the geodesic foliation $\mathcal F(v^g)$ on it, up to a homeomorphism (often a diffeomorphism). 

Also, for such $g$,  the knowledge of the scattering map $C_{v^g}$ makes it possible to recover the homology of $M$, the Gromov simplicial semi-norm on it, and the fundamental group of $M$. Additionally, $C_{v^g}$ allows to reconstruct the naturally stratified topological type of the space of geodesics on $M$. 

We aim to understand the constraints on $(M, g)$, under which the scattering data allow for a reconstruction of $M$ and the metric $g$ on it, up to a natural action of the diffeomorphism group $\mathsf{Diff}(M, \d M)$. In particular, we consider 
a closed Riemannian $n$-manifold $(N, g)$ which is locally symmetric and of negative sectional curvature. Let $M$ is obtained from $N$ by removing an $n$-domain $U$, such that the metric $g|_M$ is \emph{boundary generic}, of the gradient type, and the homomorphism $\pi_1(U) \to \pi_1(N)$ of the fundamental groups is trivial. Then we prove that  the scattering map $C_{v^{g|_M}}$ makes it possible to recover $N$ and the metric $g$ on it. 
\end{abstract} 

\section{Introduction}

\noindent Let $M$ be a compact connected and smooth Riemannian $n$-manifold with boundary. In this paper, we apply the Holographic Causality Principle (\cite{K4}, Theorem 3.1) to the geodesic flow on the space $SM$ of unit tangent vectors on $M$.  
\smallskip

Our main observation is that the holographic causality is intimately linked to the classical \emph{inverse scattering problems}. So the geodesic scattering is the focus of our present investigation. \smallskip

Let us briefly explain what we mean by the \emph{scattering data} on a given compact connected Riemannian manifold $M$ with boundary $\d M$. For each geodesic curve $\g \subset M$ which ``enters" $M$ through a point $m \in \d M$ in the direction of an unitary tangent vector $u \in T_m(M)$, we register the first along $\g$ ``exit point" $m' \in \d M$ and the exit direction, given by a unitary tangent vector $u' \in T_{m'}(M)$ at $m'$. Of course, not for any geodesic $\g$ on $M$, this construction makes sense:  $\g \setminus m$ may belong to the interior of $M$. In such case, the geodesic $\g$ through $m \in \d M$ never reaches the boundary again.  

In any case, when available, we call the correspondence $\{(m, v) \Rightarrow (m', v')\}_{(m, v)}$ ``the metric-induced scattering data".\smallskip

We strive to restore the metric $g$ on $M$, up to the action of $M$-diffeomorphisms that are the identity maps on $\d M$, from the scattering data\footnote{This resembles the problem of reconstructing the mass distribution from the gravitational lensing.}. This restoration seems harder when $(M, g)$ has closed geodesics or geodesics that originate at a boundary point, but never reach the boundary $\d M$ again.  

In special cases, the restoration of $g$ is possible. This conclusion is very much inline with the results from \cite{Cr}, \cite{Cr1}, \cite{We}, as well as with \cite{SU}  - \cite{SU4} and \cite{SUV} - \cite{SUV3}. The recent paper \cite{SUV3}, which reflects the modern state of the art, contains the strongest results.  

Recall that there are examples of two analytic Riemannian manifolds with isometric boundaries and identical scattering (even lens) data, but with different $C^\infty$-jets of the metric tensors at the boundaries (see \cite{Zh}, Theorem 4.3)! However, these examples have \emph{trapped} geodesics; the metrics there are fundamentally different from the ones we study here.

Moving towards the goal of $g$-reconstruction from the scattering data, we introduce a class of metrics $g$ which we call \emph{metrics of the gradient type} (see Definition \ref{def9.1}). By Lemma \ref{lem9.2}, the the gradient type metrics are exactly the nontrapping metrics. In Theorem \ref{th9.0}, we prove that, given any compact connected Riemannian $n$-manifold $(N, g)$ with boundary that admits a $\e$-\emph{flat triangulation}, where $\e > 0$ is a universal constant that depends only on $\dim(N)$ (see Definition \ref{def9.2a}), it is possible to delete several smooth $n$-balls $\{B_\a\}$ from $N$, so that $M = N \setminus \big(\coprod_\a B_\a\big)$ is diffeomorphic to $N$, and the restriction $g|_M$ is of the gradient type. In particular, any connected $M$ with boundary admits a gradient type Riemannian metric $g$, provided that $M$ admits a $\e$-flat triangulation for a sufficiently small $\e >0$ and a different metric $\tilde g$. The gradient type metrics form an open nonempty set $\mathcal G(M)$ in the space $\mathcal R(M)$ of all Riemannian metrics on $M$ (Corollary \ref{cor9.3}).  \smallskip

Then we introduce another class of Riemannian metrics $g$ on $M$ such that the boundary $\d M$ is ``generically curved" in $g$ (see Definition \ref{def9.3}). We call such $g$ \emph{geodesically boundary generic}, or \emph{boundary generic} for short.  We denote by $\mathcal G^\dagger(M)$ the space of geodesically boundary generic metrics of the gradient type. We speculate (see Conjecture 2.2) that, for any $M$, the space $\mathcal G^\dagger(M)$ is open and dense in $\mathcal G(M)$ and prove that it is indeed open (see Theorem \ref{th9.1}). 

We also consider a subspace $\mathcal G^\ddagger(M)\subset \mathcal G^\dagger(M)$, formed by metrics $g$ for which the geodesic vector field $v^g$ on $SM$ is \emph{traversally generic} in the sense of Definition 3.2 from \cite{K2}. Again, $\mathcal G^\ddagger(M)$ is open in $\mathcal G^\dagger(M)$.
\smallskip

In Theorem \ref{th9.2}, the main result of this paper, we prove that, for a metric $g \in \mathcal {G^\dagger}(M)$,  the geodesic flow $v^g$ on $SM$ is \emph{topologically rigid} for given scattering data. This means that, when two scattering maps, $C_{v^{g_1}}$ and $C_{v^{g_2}}$, are conjugated with the help of a smooth diffeomorphism $\Phi^\d: \d(SM_1) \to  \d(SM_2)$, then the un-parametrized geodesic flows on $SM_1$ and $SM_2$ are conjugated with the help of an appropriate homeomorphism (often a diffeomorphism) $\Phi: SM_1 \to  SM_2$ which extends $\Phi^\d$.
\smallskip

In fact, for all metrics of the gradient type, the geodesic field $v^g$ on $SM$ allows for an arbitrary accurate $C^\infty$-\emph{approximations} by boundary generic and traversing (or even traversally generic) fields $w$ on $SM$.  For such $w$, the topological restoration of the $w$-induced $1$-dimensional oriented foliation $\mathcal F(w)$ on $SM$ from the new ``scattering data" $$C_w: \d_1^+SM(w) \to \d_1^-SM(w)$$ becomes possible. However, the difficulty is to find the approximating field $w$ in the form $v^{\tilde g}$ for some metric $\tilde g$ on $M$.
\smallskip

Let $\mathcal T(v^g)$ denote the \emph{space of geodesics} in $M$. In Theorem \ref{th9.3}, we prove that,  for any metric $g \in \mathcal {G^\dagger}(M)$, the scattering data are sufficient for a reconstruction of the stratified topological type of $\mathcal T(v^g)$. In general, $\mathcal T(v^g)$ is not a smooth manifold, but for $g \in \mathcal {G^\ddagger}(M)$, it is a compact $CW$-complex \cite{K4}. For $g \in \mathcal G^\dagger(M)$, the space $\mathcal T(v^g)$ carries some ``surrogate smooth structure" \cite{K4}. This structure is also captured by the scattering data. \smallskip

In Theorem \ref{th9.4}, we prove that, for any $g \in \mathcal {G^\dagger}(M)$, the geodesic scattering map $C_{v^g}$ allows for a reconstruction of the homology spaces $H_\ast(M; \R)$ and $H_\ast(M, \d M; \R)$, equipped with the Gromov simplicial semi-norms (see \cite{G} for the definition).  In particular, the simplicial volume of the relative fundamental cycle $[M, \d M]$ can be recovered from the scattering map. 

If $\dim M \geq 3$, the geodesic scattering map $C_{v^g}$ also allows for a reconstruction of  the fundamental group $\pi_1(M)$, together with all the homotopy groups $\{\pi_i(M)\}_{i < \dim M}$. Moreover, if the tangent bundle of $M$ is trivial, $C_{v^g}$ allows for a reconstruction of the stable topological type of the manifold $M$. \smallskip

Let $(N, g)$ be a closed smooth locally symmetric Riemannian $n$-manifold, $n \geq 3$, of negative sectional curvature. In Theorem \ref{th9.6}, we prove that, if  and $M$ is obtained from $N$ by removing the interior of a smooth $n$-ball so that $\tilde g =_{\mathsf{def}} g|_M$ is of the gradient type and boundary generic, then the knowledge of the scattering map $$C_{v^{\tilde g}}: S^{n-1}\times D^{n-1} \to S^{n-1}\times D^{n-1}$$ makes it possible to reconstruct $N$ and the metric $g$, up to a positive scalar factor. However, this result does not imply the possibility of reconstructing $\tilde g$ from $C_{v^{\tilde g}}$.
\smallskip

In Section 4, we study the inverse scattering problem in the presence of additional information about the lengths of geodesic curves that connect each point $(m, v) \in \d_1^+(SM)$ to the ``scattered" point $C_{v^g}(m, v) \in \d_1^-(SM)$. This information is commonly called ``\emph{the lens data}". Our main result here, Theorem \ref{th9.10}, claims the \emph{strong} topological rigidity (see Definition \ref{def9.4}) of the geodesic flow for given lens data. The proof of the theorem requires additional hypotheses about the metric $g$, which we call ``\emph{ballanced}" (see Definition \ref{def9.?}). We apply Theorem \ref{th9.10} to the case of manifolds $M$, obtained from closed Riemannian manifolds $(N, g)$ by removing special domains $U \subset N$. Then the scattering problem on $(M, g|_M)$ is intimately linked to the geodesic flow on $SN$ (see the ``Cut and Scatter" Theorem \ref{th9.11}).  By combining Theorem \ref{th9.11} with some classical results from \cite{CK}, \cite{BCG}, we are able to reconstruct $(N, g)$ from the scattering and lens data on $M$, provided that either $N$ is locally symmetric and of a negative sectional curvature, or admits a non-vanishing parallel vector field (see Corollary \ref{cor9.14} and Corollary \ref{cor9.15}).   
\smallskip

In general, our approach to the inverse scattering problem relies more on the methods of Differential Topology and Singularity Theory, and less on the more analytical methods of Differential Geometry and Operator Theory. Of course, this topological approach has its limitations: by itself it allows only for a reconstruction of the geodesic flow from the scattering and lens data.

The assumption that the Riemannian manifolds in this study are \emph{smooth} seems to be crucial for the effectiveness of our methods. Perhaps, similar results are valid under the weaker assumption that the $n$-manifolds we investigate have a $C^{2n}$-differentiable structure.  
\section{Boundary Generic Metrics of the Gradient Type}
 
Let $M$ be a compact $n$-dimensional smooth Riemannian manifold with boundary, and  $g$ a smooth Riemannian metric on $M$. Let $SM \to M$ denote the tangent spherical bundle. With the help of $g$, we may interpret the bundle $SM$ is a subbundle of the tangent bundle $TM$.  
\smallskip

The metric $g$ on $M$ induces a partially-defined one-parameter family of diffeomorphisms $\{\psi^t_g: SM \to SM\}$, the \emph{geodesic flow}. Each unit tangent vector $u$ at a point $m \in M \setminus \d M$ determines a unique geodesic curve $\g_{(m, u)} \subset M$ trough $m$ in the direction of $u$. When $m \in \d M$, the geodesic curve $\g_{(m, u)}$ is uniquely-defined for unit vectors $u \in T_mM$  that point inside of $M$.  

By definition,  $\psi^t_g(m, u)$ is the point $(m', u') \in SM$ such that the distance along $\g_{(m, u)}$ from $m' \in \g_{(m, u)}$ to $m$ is $t$, and $u'$ is the tangent vector to $\g_{(m, u)}$ at $m'$.  We stress that $\psi^t_g(m, u)$ may not be well-defined for all $t \in \R$ and all $u \in TM$: some geodesic curves $\g$ may reach the boundary $\d M$ in finite time, and some tangent vectors $u \in TM|_{\d M}$ may point outside of $M$. However, such constraints are common to our enterprise (\cite{K}, \cite{K1} - \cite{K5}), which deals with such boundary-induced complications. 
\smallskip

In the local coordinates $(x^1, \dots, x^n,\,  p^1, \dots p^n)$ on the tangent space $TM$, the equations of the geodesic flow are:
\begin{eqnarray}\label{eq9.5}
\dot x^\a & = & \sum_\b \; g_{\a\b}\,p^\b \nonumber \\
\dot p^\a & = & -\frac{1}{2} \sum_{\b, \g} \frac{\d g_{\b\g}(x)}{\d x^\a}\,p^\b p^\g, 
\end{eqnarray}
where $g_{\a\b}(x)$ is the metric tensor.  

This system can be rewritten in terms of the Hamiltonian function $$H^g(x, p) =_{\mathsf{def}}\, \frac{1}{2} \sum_{\a, \b}\, g_{\a\b}(x) p^\a p^\b$$---the kinetic energy---in the familiar Hamiltonian form: 
\begin{eqnarray}\label{eq9.6}
\dot x^\a & = & \;\, \frac{\d H^g}{\d p^\a} \nonumber \\
\dot p^\a & = & - \frac{\d H^g}{\d x^\a}
\end{eqnarray}

The projections of the trajectories of (\ref{eq9.5}) (or of (\ref{eq9.6})) on $M$ are the geodesic curves. 
\smallskip

Let  $v^g \in T(TM)$ be the field on the manifold $TM$, tangent to the trajectories of the geodesic flow $\Psi_t^g$ on $TM$. Note that $v^g$-flow is tangent to $SM \subset TM$,  so $v^g|_{SM} \neq 0$. 

In fact, the integral trajectories of $v^g$ are \emph{geodesic curves} in the Sasaki metric $\mathsf g = \mathsf g(g)$ on $T(TM)$ (\cite{Be}, Prop. 1.106).
\smallskip

Let $X$ be a compact smooth $(m+1)$-manifold with boundary. Any smooth vector field $v$ on $X$, which does not vanish along the boundary $\d X$, gives rise to a partition $\d_1^+X(v) \cup \d_1^-X(v)$ of the boundary $\d X$ into  two sets: the locus $\d_1^+X(v)$, where the field is directed inward of $X$ or is tangent to $\d X$, and  $\d_1^-X(v)$, where it is directed outwards or is tangent to $\d X$. We assume that $v|_{\d X}$, viewed as a section of the quotient  line bundle $T(X)/T(\d X)$ over $\d X$, is transversal to its zero section. This assumption implies that both sets $\d^\pm_1 X(v)$ are compact manifolds which share a common boundary $\d_2X(v) =_{\mathsf{def}} \d(\d_1^+X(v)) = \d(\d_1^-X(v))$. Evidently, $\d_2X(v)$ is the locus where $v$ is \emph{tangent} to the boundary $\d X$.

Morse has noticed (see \cite{Mo}) that, for a generic vector field $v$, the tangent locus $\d_2X(v)$ inherits a similar structure in connection to $\d_1^+X(v)$, as $\d X$ has in connection to $X$. That is, $v$ gives rise to a partition $\d_2^+X(v) \cup \d_2^-X(v)$ of  $\d_2X(v)$ into  two sets: the locus $\d_2^+X(v)$, where the field is directed inward of $\d_1^+X(v)$ or is tangent to $\d_2X(v)$, and  $\d_2^-X(v)$, where it is directed outward of $\d_1^+X(v)$ or is tangent to $\d_2X(v)$. Again, we assume that $v|_{\d_2X(v)}$, viewed as a section of the quotient  line bundle $T(\d X)/T(\d_2X(v))$ over $\d_2X(v)$, is transversal to its zero section.

For, so called, \emph{boundary generic} vector fields (see \cite{K1} for a formal definition), this structure replicates itself: the cuspidal locus $\d_3X(v)$ is defined as the locus where $v$ is tangent to $\d_2X(v)$; $\d_3X(v)$ is divided into two manifolds, $\d_3^+X(v)$ and $\d_3^-X(v)$. In  $\d_3^+X(v)$, the field is directed inward of $\d_2^+X(v)$ or is tangent to its boundary, in  $\d_3^-X(v)$, outward of $\d_2^+X(v)$ or is tangent to its boundary. We can repeat this construction until we reach the zero-dimensional stratum $\d_{m+1}X(v) = \d_{m+1}^+X(v) \cup  \d_{m+1}^-X(v)$.  

To achieve some uniformity in the notations, put $\d_0^+X =_{\mathsf{def}} X$ and $\d_1X =_{\mathsf{def}} \d X$. 

Thus a boundary generic vector field $v$  on $X$  gives rise to two stratifications: 
\begin{eqnarray}\label{eq2.0}
\d X =_{\mathsf{def}} \d_1X \supset \d_2X(v) \supset \dots \supset \d_{m +1}X(v), \nonumber \\ 
X =_{\mathsf{def}} \d_0^+ X \supset \d_1^+X(v) \supset \d_2^+X(v) \supset \dots \supset \d_{m+1}^+X(v),  \nonumber 
\end{eqnarray}
the first one by closed smooth submanifolds, the second one---by compact ones.  Here $\dim(X) = m+1$, and $\dim(\d_jX(v)) = \dim(\d_j^+X(v)) = m +1 - j$. 

We will use often the notation ``$\d_j^\pm X$" instead of ``$\d_j^\pm X(v)$" when the vector field $v$ is fixed or its choice is obvious. \smallskip

As any non-vanishing vector field, the geodesic field $v^g \in T(SM)$ divides the boundary $\d(SM)$ into two portions: $\d_1^+(SM)$, where $v^g$ points inside of $SM$ or is tangent to its boundary, and $\d_1^-(SM)$, where it points outside of $SM$ or is tangent to its boundary. In fact, $\d_1^+(SM)$ and $\d_1^-(SM)$ do not depend on $g$: the first locus is formed by the pairs $(m, v)$, where $m \in \d M$ and $v \in T_mM$ points inside of $M$ or is tangent to $\d M$. Therefore, both $\d_1^+(SM)$ and $\d_1^-(SM)$ are homeomorphic to the tangent $(n-1)$-disk bundle of the manifold $\d M$. The locus $\d_2(SM) = \d(\d_1^+(SM)) = \d(\d_1^-(SM))$ is also $g$-independent; it is the space of the sphere bundle, associated with the tangent $(n-1)$-bundle $T(\d M)$.

\begin{definition}\label{def9.1} Let $(M, g)$ be a compact connected smooth Riemannian manifold with boundary.

We say that a metric $g$ on $M$ is of the \emph{gradient type} if the vector field $v^g \in T(SM)$ that governs the geodesic flow is of the gradient type: that is, there exists a smooth function $F: SM \to \R$ such that $dF(v^g) > 0$. 

This condition is equivalent to the property $\{\tilde F, H^g\} > 0$, where $\tilde F: TM \setminus M \to \R$ is a smooth extension of $F$, the Hamiltonian $H^g$ is defined by equations (\ref{eq9.5}) and (\ref{eq9.6}), and $\{ , \}$ stands for the Poisson bracket of functions on $TM$. 
\smallskip

We denote by $\mathcal G(M)$ the space of the gradient-type Riemannian metrics on $M$. \hfill $\diamondsuit$
 \end{definition}
 
\begin{figure}[ht]\label{fig9.2}
\centerline{\includegraphics[height=1.8in,width=2in]{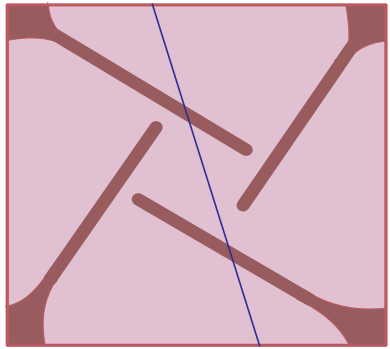}}
\bigskip
\caption{\small{The flat metric on torus becomes of the gradient type on the complement to a curvy disk (whose ``center" is at the corners of the square fundamental domain).}} 
\end{figure}

\smallskip

\noindent {\bf Example 2.1.} 
Consider a flat metric $g$ on the torus $T^2 = \R^2/\Z^2$ and form a punctured torus $M$ by removing an open disk $D^2$ from $T^2$. If $D^2$ is convex in the fundamental square domain $Q^2 \subset \R^2$, then there exist closed geodesics (with a rational slope with respect to the lattice $\Z^2$) that miss $D^2$. For such $M \subset T^2$, the flat metric is not of the gradient type. 

However, it is possible to position $D^2 \subset T^2$ so that its lift $\tilde D^2$ to $\R^2$ will have intersections with \emph{any} line that passes through $Q^2$ (see Figure 1).  
We restrict the flat metric $g$ to $M = T^2 \setminus D^2$. For such a choice of $(M, g|_M)$, thanks to Lemma \ref{lem9.2} below, the metric $g|_M$ is of the gradient type. Moreover, by Theorem \ref{th9.1}, any metric $g'$ on $M$, sufficiently close to this flat metric $g|_M$, is also of the gradient type. \hfill $\diamondsuit$

\begin{lemma}\label{lem9.1} The set $\mathcal G(M)$ of Riemannian metrics $g$ of the gradient type is \emph{open} in the space $\mathcal R(M)$ of all Riemannian metrics on $M$, considered in the $C^\infty$-topology. 
\end{lemma}

\begin{proof} Let $\s: M \to TM$ denote the zero section. If $dF(v^g) > 0$ on $SM$, then $F: SM \to \R$ extends in a compact neighborhood $U$ of $SM$ in $TM \setminus \s(M)$ to a smooth function $\tilde F: U \to \R$ so that $d\tilde F(v^g) > 0$ in $U$. In this neighborhood, $d\tilde F(v^{g'})|_U > 0$ for all metrics $g'$, sufficiently close to $g$. For such metrics $g'$, the space of unit spheres $S'M \subset TM$ is \emph{fiberwise} close to $SM \subset TM$; in particular, we may assume that $S'M \subset U$. Recall that  the geodesic field $v^{g'}$ is tangent to $S'M$, thus $d\tilde F(v^{g'}) > 0$ on $S'M$.
\end{proof}

\begin{definition}\label{def9.1a} A Riemannian metric $g$ on a connected compact manifold $M$ with boundary is called \emph{non-trapping} if $(M, g)$ has no closed geodesics and no geodesics of infinite length (the later are homeomorphic to an open or a semi-open interval). \hfill $\diamondsuit$ 
\end{definition}
\smallskip

\begin{lemma}\label{lem9.2} Let $M$ be a compact connected smooth manifold with boundary. A metric $g$ on $M$ is of the gradient type, if and only if, any trajectory of the geodesic flow is homeomorphic to a closed interval or to a singleton. 

In other words, the non-trapping metrics and the metrics of the gradient type are the same.
 \end{lemma}

\begin{proof} If, for a smooth function $F$, $d F(v^{g}) > 0$, then each $v^g$-trajectory is singleton residing in $\d(SM)$ or a closed segment with its both ends residing in $\d(SM)$  (we call such vector fields \emph{traversing}). 

Evidently, $F$ prevents $v^g$ from having a closed trajectory. Let $\tilde\g$ be a trajectory that starts at a point $w \in SM$ and is homeomorphic to a semi-open interval. So $\tilde\g $ extends beyond any point on $\tilde\g$ that can be reached from $w$ ($\tilde\g$ cannot ``exit" $SM$ in a finite time). Consider the closure $K$ of $\tilde\g$. It is a compact and $v^g$-invariant set. So $F$ attends its maximum at a point $w_\star \in K$. However, $dF(v^g) > 0$ at $w_\star$ and a germ of a $v^g$-trajectory through $w_\star$ belongs to $K$, a contradiction to the assumption that $F$ attends its maximum at $w_\star$ .  

A similar argument rules out the $(-v^g)$-trajectories that are homeomorphic to a semi-open interval. \smallskip 

Conversely, by Lemma 5.6 from \cite{K1}, any traversing field is of the gradient type. So $v^g$ admits a Lyapunov function $F$.   

Thus $g$ is of the gradient type if and only if the image of any $v^g$-trajectory $\tilde\g$ under the map $SM \to M$ is either a singleton in $\d M$ or a compact geodesic curve $\g$ whose ends reside in $\d M$ (by our convention, $\g$ does not extends beyond its two ends). In particular, any $g \in \mathcal G(M)$ has no closed geodesics in $M$ and no geodesics that originate at the boundary $\d M$ and are trapped in $\textup{int}(M)$ for all positive times.
\end{proof}

\begin{corollary}\label{cor9.1} Let $(N, g)$ be an open Riemannian manifold such that no geodesic curve in $N$ is closed or has an end that is contained in a compact set. Let $M \subset N$ be a smooth compact codimension zero submanifold. Then the restriction $g|_M$ is a metric of the gradient type, and so are all the metrics $\tilde g$ on $M$ that are sufficiently close to $g|_M$. \smallskip

In particular, for any compact domain $M$ with a smooth boundary in the Euclidean space $\mathsf E^n$ or in the hyperbolic space $\mathsf H^n$, the Euclidean metric  $g_{\mathsf E}$ or the hyperbolic metric $g_{\mathsf H}$ on $M$ are of the gradient type, and so are all the metrics $\tilde g$ that are sufficiently close to $g_{\mathsf E}$ or $g_{\mathsf H}$, respectively.
\end{corollary}

\begin{proof} Using the hypotheses, no positive time geodesic $\g \subset M$ in the metric $g|_M$ is an image of a semi-open interval or a closed loop. By Lemma \ref{lem9.2}, the pair $(M, g|_M)$ is of the gradient type. 

By Lemma \ref{lem9.1}, any metric $\tilde g$ on $M$, which is sufficiently close to $g|_M$, is of the gradient type as well.
\end{proof}

In order to prove Theorem \ref{th9.0} below, we will need few lemmas, dealing with smooth triangulations of compact smooth Riemannian manifolds, triangulations that are specially adjusted to the given metric. \smallskip

Let $N$ be a smooth compact $n$-manifold. A smooth triangulation $T: K \to N$ is a homeomorphism from a finite simplicial complex $K$ to $N$. The triangulation is assembled out of several homeomorphisms $\{T_j:  \D \to N\}_j$, where $\D$ denotes of the standard $n$-simplex $\D \subset \R^{n+1}$, the restriction of $T_j$ to the interior of each subsimplex  $\D' \subset \D$ being a smooth diffeomorphism.  The homeomorphisms  $\{T_j:  \D \to N\}_j$ commute with affine maps of subsimplicies $\D' \subset \D$, the maps that assemble $K$ out of several copies of $\D$.

\begin{definition}\label{def9.2a}  
Consider a smooth triangulation $T: K \to N$ of a smooth compact $n$-manifold $N$ with a Riemannian metric $g$. 

Let $\g \subset T_j(\D)$ be a geodesic arc, and $\tilde\g_j =_{\mathsf{def}}(T_j)^{-1}(\g) \subset \D$ its preimage in the standard $n$-simplex $\D$. We denote by $l_{\mathsf E}(\tilde\g_j)$ its length in the Euclidean metric on $\D \subset \R^{n+1}$.
Let $\lambda(\tilde\g_j) \subset \D$ be a line segment that shares its ends with $\tilde\g_j $. We denote by $l_{\mathsf E}(\lambda(\tilde\g_j))$ its length.

Pick a number $\e> 0$. We say that a smooth triangulation $T: K \to N$ is $\e$-\emph{flat} with respect to $g$ if, for each index $j$ and any geodesic arc $\g \subset T_j(\D)$, the inequality is valid:
$$l_{\mathsf E}(\tilde\g_j) < (1+\e)\cdot l_{\mathsf E}(\lambda(\tilde\g_j)).$$  
\hfill $\diamondsuit$
\end{definition}

Note that the inequality in this definition remains valid under the conformal scaling of the simplex $\D$.  

\begin{lemma}\label{lem9.2a} Let $N$ be a compact  smooth $n$-manifold, equipped with a Riemannian metric $g$, and $U$ a convex domain in $\R^n$, equipped with the Euclidean metric $g_{\mathsf E}$. Consider a diffeomorphism $S: U \to S(U) \subset N$. Let $\g  \subset S(U)$ be a geodesic arc, and $\tilde\g$ its $S$-preimage. We denote by $\lambda(\tilde\g)$ the segment in $U$ that shares its ends with the arc $\tilde\g$. 

Assume that, for some $\e > 0$ and any geodesic arc $\g \subset S(U)$, the Euclidean lengths of $\tilde\g$ and $\lambda(\tilde\g)$ satisfy the inequality 
$$l_{\mathsf E}(\tilde\g) < (1+\e)\cdot l_{\mathsf E}(\lambda(\tilde\g)).$$

Then the arc $\tilde\g$ is contained in the $\delta$-neighborhood of the segment $\lambda(\tilde\g)$, where $\delta = l_{\mathsf E}(\lambda(\tilde\g))\cdot \sqrt{\frac{\e}{2}}\cdot \sqrt{1+\frac{\e}{2}}$.
\end{lemma}

\begin{proof} Let $\lambda  =_{\mathsf{def}} \lambda(\tilde\g)$ and let $\tilde a, \tilde b$ be the two ends of the segment $\lambda$. Put $\ell = d_{\mathsf E}(\tilde a, \tilde b) = l_{\mathsf E}(\lambda(\tilde\g))$. Consider the solid ellipsoid  $$\mathcal E = \{\tilde c \in U| d_{\mathsf E}(\tilde a, \tilde c) + d_{\mathsf E}(\tilde c, \tilde b) < (1+\e)\ell\}.$$  
The maximal distance from the main axis $\mu$ of the ellipsoid $\mathcal E$ to its boundary is the radius $\delta$ of the $(n-2)$-sphere, obtained by intersecting the bisector hyperplane $H$, orthogonal to $\lambda$ at its midpoint, with  $\d\mathcal E$. From the elementary $2D$-geometry, we get $\delta^2 + (\ell/2)^2 = ((1+\e)(\ell/2))^2$. Thus $\delta = \ell\cdot \sqrt{\frac{\e}{2}}\cdot \sqrt{1+\frac{\e}{2}}$. As a result, $\tilde\g$ must be contained in the $\delta$-neighborhood of $\mu$ ($\mu$ contains the segment $\lambda$). Additional elementary computations show that $\mathcal E$ is contained in the $\delta$-neighborhood of $\lambda$ as well.
 
 Take a typical point $x \in \tilde\g$. By the hypotheses, $l_{\mathsf E}(\lambda(\tilde\g)) < (1+\e)\ell$. Thus $$(1+\e)\ell > l_{\mathsf E}(\lambda(\tilde\g)) \geq d_{\mathsf E}(\tilde a, x) + d_{\mathsf E}(x, \tilde b).$$ 
Therefore $x \in \mathcal E$. As a result, $\tilde\g$ is contained in the neighborhood of $\lambda$ of the radius $\delta$.
 \hfill
 \end{proof}

\begin{figure}[ht]\label{fig9.3}
\centerline{\includegraphics[height=2in,width=3.7in]{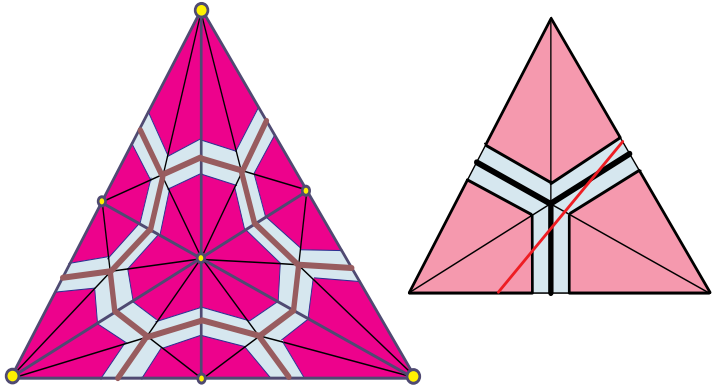}}
\bigskip
\caption{\small{The sets $St(\theta) \cap \D$ (left diagram) and $St(\theta) \cap \D' \subset \b\D$ (right diagram)}} 
\end{figure}

Let $\D$ be the standard $n$-simplex, residing in $\R^{n+1}$ and equipped with the Euclidean metric. We denote by $\b\D$ and $\b^2\D$ the first and the second barycentric subdivisions of $\D$. For any vertex $a \in \b\D$, we form its star $St(a)$ in $\b^2\D$. For any $0 < \theta \leq 1$, we denote by $St(a, \theta) \subset \D$  the $\theta$-homothetic image of $St(a)$, the center of homothety being at $a \in \b\D$. Consider the set $$St(\theta) =_{\mathsf{def}}\coprod_{a \in \b\D} St(a, \theta).$$

By a \emph{line} in $\D$ we mean an intersection of an affine line in $\R^n \subset \R^{n+1}$ with $\D$.

\begin{lemma}\label{lem9.3c} There exists a number $\theta_\star(n) \in (0, 1)$ such that every line $\mathcal L \subset \D$ has a point $a$ that belongs to the interior of $St(\theta_\star(n))$.\footnote{It is desirable to find an elementary argument that explicitly computes $\theta_\star(n)$ as a function of $n$.}
\end{lemma}

\begin{proof} Let $\D'$ be a typical simplex of $\b\D$. It will suffice to show that there exists an universal $\theta_\star(n) \in (0, 1)$, such that any line $\mathcal L \subset \D'$ has a point $b$ that belongs to the interior of some set $St(a, \theta_\star(n)) \cap \D'$, for a vertex $a \in \b\D$.

For each $\theta \in (0, 1]$, consider the polyhedron $P(\theta) =_{\mathsf{def}} \D \setminus \textup{int}(St(\theta))$. For each simplex $\D'$ of $\b\D$, put $P_{\D'}(\theta) = P(\theta) \cap \D'$. 

The space of lines $\mathsf L$ in $\D'$ is compact; in fact, it is a continuous image of a compact subset $\mathsf K$ of the Grassmanian $\mathsf{Gr}(n+1, 2)$. $\mathsf K$ consists of the $2$-planes through the point $(1, 0, \dots , 0) \in \R^{n+1}$ that have a nonempty intersections with $\D' \subset \R^n$. Here the hyperplane $\R^n \subset \R^{n+1}$ is defined by equating the first coordinate with zero. This construction gives rise to a continuous map $\pi: \mathsf K \to \mathsf L$. Any line whose intersection with $\D' \subset \R^n$ is not a singleton (equivalently whose intersection with $\D'$ is not a vertex of $\D'$), defines a unique point in $\mathsf{Gr}(n+1, 2)$.  

Consider an increasing sequence $\{\theta_i \in (0, 1)\}_i$ that converges to 1.   Contrarily to the claim of  the lemma, assume  that for each $i$, there exists a line $\mathcal L_i$ that is contained in the polyhedron $P_{\D'}(\theta_i)$. 

Using compactness of $\mathsf L$, there exists a subsequence $\{\mathcal L_{i_k}\}_k$ that converges to a limiting line $\mathcal L \subset \D'$. Then $\mathcal L \subset \bigcap_k P_{\D'}(\theta_{i_k}) = P_{\D'}(1)$.  Note that $\mathcal L$ is missing the verticies of $\D'$. On the other hand, by the construction of $\b\D'$, if a line $\mathcal L \subset \D'$ has a pair of distinct points $b, c$ such that the segment $[a, b] \subset P_{\D'}(1)$, then $\mathcal L \cap \textup{int}(St(1)) \neq \emptyset$. 

This contradiction proves that exists $\theta_\star(n) \in (0, 1)$ for which no line in $\D'$ is contained in $P(\theta_\star(n))$. The definition of the polyhedron $P_{\D'}(\theta_\star(n)) \subset P(\theta_\star(n))$ is given by an affine (metric-independent) construction. So the polyhedra $P_{\D'}(\theta_\star(n))$ for different $\D' \subset \D$ match automatically. By the same token, $P(\theta_\star(n))$ does not contain any lines for any simplex $\D$, not necessarily for the standard one.

Thus there is a number $\theta_\star(n) < 1$ such that, for each $\theta \in [\theta_\star(n), 1]$, every line $\mathcal L \subset \D$ hits the set $St(\theta).$
\end{proof}

\begin{conjecture}\label{conj9.0} Let $N$ be a compact smooth Riemannian manifold. For any sufficiently small $\e > 0$, $N$ admits a smooth $\e$-flat triangulation $T: K \to N$. \hfill $\diamondsuit$
\end{conjecture}

\begin{theorem}\label{th9.0} Let $\theta_\star(n) \in (0,1)$ be as in Lemma \ref{lem9.3c}. Put $\theta'_\star(n) =_{\mathsf{def}} (1 + \theta_\star(n))/2$, and let $\e_\star(n)$ denote the Euclidean distance between the sets $St(\theta_\star(n))$ and $P(\theta'_\star(n))$ in the standard simplex $\D$. 
\begin{itemize}
\item Let $(N, g)$ be a closed connected smooth Riemannian $n$-manifold that admits a $\e_\star(n)$-flat smooth triangulation\footnote{By Conjecture \ref{conj9.0}, any $(N,g)$ will do.}. Then there exists a smooth $n$-ball $B \subset \textup{int}(N)$ such that the restriction of the metric $g$ to $M = N \setminus \textup{int}(B)$ is of the gradient type. \smallskip

\item If $(N, g)$ is a compact connected Riemannian $n$-manifold with boundary that admits a $\e_\star(n)$-flat smooth triangulation. Then, for each connected component $\d_\a N$ of the boundary $\d N$, there exists a relative $n$-ball $(B_\a, \delta B_\a) \subset (N, \d_\a N)$, where  $\delta_\a B =_{\mathsf{def}} B \cap \d_\a N$ is the $(n-1)$-ball, so that all the balls are disjointed and the restriction of the metric $g$ to the manifold $M = N \setminus \coprod_\a \textup{int}(B_\a)$ is of the gradient type. The manifolds $M$ and $N$ are diffeomorphic.
\end{itemize}
\end{theorem}

\begin{proof} The idea is first to construct a number of disjointed balls in $N$ so that each geodesic curve will hit some ball. Thus deleting such balls from $N$ will produce a ``geodesically traversing swiss cheese". When $N$ is closed, we will incorporate all the balls into a single smooth one. When $N$ has a boundary, then the balls will be incorporated in a domain, whose removal from $N$ does not change the smooth topology of $N$.\smallskip

Let $\b T$ and $\b^2T$ denote the first and second barycentric subdivisions of a given smooth triangulation $T: K \to N$. 
As before, $T$ is assembled from a collection of singular simplicies $\{T_j: \D \to N\}_j$.
\smallskip

Put $\e_\star = \e_\star(n)$,  $\theta_\star = \theta_\star(n)$. By the hypotheses, there exists a smooth $\e_\star$-flat smooth triangulation $T: K \to N$. 

Then, for any geodesic curve $\g \subset N$,  its $T_j$-preimage $\tilde\g_j$ in the simplex $\D$ is contained in the $\e_\star$-neighborhood of a line $\mathcal L_j \subset \D_j$. By Lemma \ref{lem9.3c}, that line contains a point $a_j \in St(\theta_\star)$ whose $\e_\star$-neighborhood $U_j \subset \D$ is contained in the set $St(\theta'_\star)$. So, by Definition \ref{def9.2a}, the curve $\tilde\g_j$ must intersect the neighborhood $U_j$. As a result, $\g$ has a non-empty intersection with the set $T(St(\theta'_\star))$, a finite disjoint union of $n$-dimensional $\mathsf{PL}$-balls $\{B_v\}_{v \in \b T}$, centered on the vertices of $\{T(v)\}_{v \in \b T}$ in $N$.

We can smoothen their boundaries by encapsulating each ball $B_v$ into a smooth ball $\hat B_v$ so that $\hat B_v \cap \hat B_{v'} = \emptyset$ for all distinct $v, v' \in \b T$. 

When $N$ is closed,  we place the disjoint union $A =_{\mathsf{def}} \coprod_{v \in \b T} \hat B_v$ inside of a single smooth ball $B \subset N$. This may be accomplished by attaching $1$-handles to $A$ so that the cores of the handles form a tree. Any geodesic in $N$ hits $B$ since it hits $A$.

In the case of a non-empty boundary $\d N$, by attaching first some relative $1$-handles $\{H_i \approx D^{n-1}_+ \times [0, 1]\}$, whose cores $\{0 \times [0, 1] \}$ reside in $\d N$, transforms $A \cap \d N$ into a disjoint union of several $(n-1)$-balls, each ball residing in its connected component of $\d N$. Again, in each component $\d_\a N$ of $\d N$, the attaching the handles is guided by a tree. 

In the process, we incapsulate $A$ into a disjoint union of several smooth balls that reside in the interior of $N$ and several relative balls, each of which is touching the corresponding boundary component $\d_\a N$. These relative balls are in 1-to-1 correspondence with the boundary components. Then connecting the balls in the interior on $N$ to the balls that touch the boundary $\d N$ by $1$-handles (which reside in $N$) produces the desired relative pairs $\{(B_\a, B_\a \cap \d N)\}_\a$, one pair per component of $\d N$. Again, any geodesic in $N$ hits the disjoint union of these pairs. The removal of the union from $N$ results in a smooth manifold $M$, which is diffeomorphic to $N$.
\end{proof}

\begin{corollary}\label{cor9.3} 
If a compact connected smooth Riemannian $n$-manifold $(M, g)$ with boundary admits a $\e_\star(n)$-flat triangulation, then $M$
admits a Riemannian metric $\tilde g$ of the gradient type. 

In fact, the subspace $\mathcal G(M)$ of such metrics $\tilde g$ is nonempty and open in the space $\mathcal R(M)$ of all Riemannian metrics on $M$.
\end{corollary}

\begin{proof} Use Theorem  \ref{th9.0} to conclude that $\mathcal G(M)$ is a nonempty set. By Lemma \ref{lem9.1}, $\mathcal G(M)$ is open in the space $\mathcal R(M)$.
\end{proof}

Note that the smooth balls, whose removal from $N$ (the ``geodesic Swiss cheese") delivers, by Theorem \ref{th9.0}, a metric of the gradient type on $M$, are not necessarily \emph{convex} in the original metric $g$.

\begin{conjecture}\label{conj9.0}  For any compact smooth Riemannian manifold $(N, g)$, there exists a finite disjoint union of smooth \emph{convex} balls whose removal from $N$ delivers the metric $g_M$ of the gradient type on their complement $M$. \hfill $\diamondsuit$
\end{conjecture}
\smallskip

\begin{definition}\label{def9.3} Let $(M, g)$ be a compact connected Riemannian manifold with boundary.
\begin{itemize}
\item We say that a metric $g$ on $M$ is \emph{geodesically  boundary generic}  if the geodesic vector field $v^g \in T(SM)$ is  boundary generic\footnote{See the discussion that precedes Definition \ref{def9.1} or Definition 2.1 from \cite{K1}.} with respect to the boundary $\d(SM) = SM|_{\d M}$.
\smallskip 

\item We say that a metric $g$ on $M$ is \emph{geodesically traversally  generic}  if the vector field $v^g \in T(SM)$ is of the gradient type and is traversally  generic\footnote{See Definition 3.2 from \cite{K2}.} with respect to $\d(SM) = SM|_{\d M}$.
\end{itemize}

We denote the space of all gradient type metrics on $M$ by the symbol $\mathcal{G}(M)$, the space of all  geodesically boundary generic metrics of the gradient type on $M$ by the symbol $\mathcal{G}^\dagger(M)$, and the space of all geodesically traversally generic metrics on $M$ by the symbol $\mathcal G^\ddagger(M)$. So we get $\mathcal G^\ddagger(M) \subset \mathcal G^\dagger(M) \subset \mathcal{G}(M)$. \hfill $\diamondsuit$ 
\end{definition}

\noindent {\bf Remark 2.2.}
If $(M, g)$ is such that there exists a geodesic curve $\g \subset M$ whose arc is contained in $\d M$, then the metric $g$ is not geodesically boundary generic. 

For example, the Euclidean metric on $\R^3$ is not geodesically boundary generic with respect to the ruled surface $S = \{z = x^2 - y^2\}$: indeed, $S$ is comprised of lines (geodesics). 

\hfill $\diamondsuit$
\smallskip

\noindent {\bf Example 2.2.}
Let $M$ be a domain in the Euclidean plane $\R^2$, bounded by simple smooth closed curves. Then the flat metric $g_{\mathsf E}$ is boundary generic on $M$ if and only if $\d M$ is comprised of strictly concave and convex loops or arcs that are separated by the cubic inflection points. In particular, no line, tangent to the boundary, has the order of tangency that exceed 3 (see Example 3.1 for the details).
\hfill $\diamondsuit$
\smallskip

\begin{question}\label{q9.1} 
\emph{How to formulate the property of the geodesic vector field $v^g$ being boundary generic/traversally generic with respect to $\d(SM)$ in terms of the geodesic curves and Jacobi fields in $M$ and their interactions with $\d M$?} \hfill $\diamondsuit$
\end{question}
\smallskip
 
\noindent {\bf Remark 2.3.}
Of course, not any metric $g$ on $M$ is of the gradient type. At the same time, thanks to Theorem \ref{th9.0}, the gradient-type metrics form a massive set. 
\smallskip 

Examples of geodesically  boundary generic metrics are also not so hard to exhibit.  They require only a \emph{localized} control of the geometry of $\d M$ in terms of $g$ (see Lemmas \ref{lem9.5} and \ref{lem9.5a}). For instance, if all the components of $\d M$ are either \emph{strictly convex} or \emph{strictly concave} in $g$, then $g$ is geodesically  boundary generic. 

In contrast, to manufacture a geodesicly traversally  generic metric is a more delicate task. In fact, we know only few examples, where gradient type metrics are proven to be of the traversally generic type: these examples have gradient-type metrics in which the boundary $\d M$ is strictly convex (see Corollary \ref{cor9.8}). However, we suspect that traversally  generic metrics are abundant (see Conjecture \ref{conj9.1}). In any case, by Theorem \ref{th9.1} below, the property of a metric $g$ to be traversally  generic is \emph{stable} under small smooth perturbations of $g$.
\hfill $\diamondsuit$
\smallskip

So we have only a weak evidence for the validity of following conjecture; however, the world in which it is valid seems to be a pleasing place... 

\begin{conjecture}\label{conj9.1} The sets $\mathcal G^\dagger(M)$ and $\mathcal G^\ddagger(M)$ are open and \emph{dense} in the space $\mathcal{G}(M)$. 
\hfill $\diamondsuit$
\end{conjecture}

The openness of $\mathcal G^\dagger(M)$ and $\mathcal G^\ddagger(M)$ in $\mathcal{G}(M)$ follows from the theorem below.
\smallskip

\begin{theorem}\label{th9.1} Let $M$ be a compact smooth connected manifold with boundary. 

In the space $\mathcal R(M)$ of all Riemannian metrics on $M$, equipped with the $C^\infty$-topology, the spaces $\mathcal G^\dagger(M)$ and $\mathcal G^\ddagger(M)$ are open. Each of these spaces is invariant under the natural action of the smooth diffeomorphism group $\mathsf{Diff}(M)$ on $\mathcal R(M)$. 
\smallskip 

If a metric $g$ on $M$ is of the gradient type, then the geodesic field $v^g$ on $SM$ can be approximated arbitrary well in the $C^\infty$-topology by a traversally  generic field $\tilde v \in \mathcal V^\ddagger(SM)$. 
\end{theorem}

\begin{proof} The construction of the geodesic flow $g \Rightarrow v^g$ defines a continuous map $$\mathcal F: \mathcal R(M) \to \mathcal V(SM),$$ where $\mathcal V(SM)$ denotes the space of all vector fields on $SM$. By Theorem 6.7   and Corollary 6.4 from \cite{K2}, the subspace $\mathcal V^\ddagger(SM)$, formed by traversally  generic (and thus  gradient-like) vector fields, is open in $\mathcal V(SM)$. Similarly, the boundary generic and traversing fields form an open set $\mathcal V^\dagger(SM) \cap \mathcal V_{\mathsf{trav}}(M)$ in $\mathcal R(M)$. 

Since the germ of the geodesic $\g$ through a point $m \in M$ in the direction of a given unit tangent vector $u$ depends smoothly on metric $g$, we conclude that $\mathcal F$ is a continuous map. Therefore, $$\mathcal G^\ddagger(M) =_{\mathsf{def}} \mathcal F^{-1}\big(\mathcal V^\ddagger(SM)\big)$$ and $$\mathcal G^\dagger(M) =_{\mathsf{def}} \mathcal F^{-1}\big(\mathcal V^\dagger(SM) \cap \mathcal V_{\mathsf{trav}}(SM)\big)$$ are open sets in $\mathcal R(M)$.  
\smallskip

By definition, for any $g \in  \mathcal G(M)$, the geodesic field $v^g$ on $SM$ is of the gradient type (and thus traversing). Again, by Theorem 6.7 from \cite{K2}, $v^g$ can be approximated by a traversally  generic field $\tilde v \in \mathcal V^\ddagger(SM)$. Note however that the projections of $\tilde v$-trajectories under the map $SM \to M$ may not stay $C^\infty$-close to the geodesic lines in the original metric $g$ due to the concave boundary effects.

By Theorem \ref{th9.0}, $\mathcal V^\dagger(SM) \neq \emptyset$. Nevertheless, the question whether $\mathcal G^\ddagger(M) \neq \emptyset$ for a given $M$ remains open!
\smallskip

Evidently, by the ``naturality" of the geodesic flow,  the spaces $\mathcal G^\dagger(M)$, $\mathcal G^\ddagger(M)$ are invariant under the natural action of the smooth diffeomorphism group $\mathsf{Diff}(M)$ on $\mathcal R(M)$.
\end{proof}

\noindent {\bf Remark 2.4.} 
Let $M$ be a codimension $0$ compact submanifold of a compact  Riemannian manifold $N$ such that $M \subset \mathsf{int}(N)$. If a metric $g$ on the ambient $N$ is of the gradient type, then, by Corollary \ref{cor9.3}, its restriction $g|_M$  is of the gradient type on $M$. 

Of course, if $g$ is geodesically traversally  generic on a compact manifold $N$, it may not be geodesically  traversally  generic on $M$. 
\hfill $\diamondsuit$ 
\smallskip

\noindent {\bf Example 2.3.}
Consider  the hyperbolic space $\H ^n$ with its virtual spherical boundary $\d \H^n$ and hyperbolic metric $g$. The space $\H^n$ is modeled by the open unit ball in the Euclidean space $\mathsf E^n$. 

Each geodesic line hits $\d \H^n$ at a pair of points, where it is orthogonal (in the Euclidean metric) to $\d \H^n$. For each oriented geodesic line $\g$ trough a given point $x \in \H^n$ in the direction of a given vector $w$, consider the  distance $d_\H(x, w)$ between $x$ and the unique point $y \in \bar\g \cap \d \H^n$ that can be reached from $x$ by moving along $\g$ in the direction of $-w$, that is, $d_\H(x, w)$ is the length of the circular arc $(y, x) \subset \bar\g$ in $\mathsf E^n$. Evidently, $d_\H(x, w)$ is strictly increasing, as one moves along the oriented $\g$.  

Let $M \subset \H^n$ be a compact  codimension $0$  smooth submanifold, equipped with the induced hyperbolic metric. Then the geodesic field $v^g$ on the space $SM$ is of the gradient type, since $d_\H(x, w)$ is strictly increasing along the oriented trajectories of $v^g$. 

Again, by Theorem \ref{th9.1}, any metric $g'$ on $M$, sufficiently close to the hyperbolic metric $g|_M$, is also of the gradient type.
\hfill $\diamondsuit$
\smallskip

\section{The Geodesic Scattering and Holography}

\noindent In this section, we will apply the Holographic Causality Principle \cite{K4}, to geodesic flows on the spaces $SM$ of unit tangent vectors on compact smooth Riemannian manifolds $M$ with boundary. 
\smallskip

We will be guided by a single important observation: if a metric $g$ on $M$ is of the gradient type, then the causality map $$C_{v^g}: \d_1^+(SM) \to \d_1^-(SM),$$ introduced in \cite{K4} (for generic smooth traversing vector fields $v$ on compact manifolds $X$ with boundary), is available! To get a feel for the nature of the causality map from \cite{K4}, the reader may glance at Figure 3. It depicts the causality map $C_v$ for a traversing field $v$ on a surface $X$ with boundary.\smallskip

The map $C_{v^g}$ represents the $g$-induced \emph{geodesic scattering}: indeed, with the help of $C_{v^g}$, each unit tangent vector $u \in T^+M|_{\d M}$ is mapped (``scattered") to a unit tangent vector $u' \in T^-M|_{\d M}$. Here $T^\pm M|_{\d M} \subset TM|_{\d M}$ denote the half-spaces, formed by vectors along $\d M$ that are tangent to $M$ and point inside/outside of $M$. 
\smallskip

\begin{figure}[ht]\label{fig9.4}
\centerline{\includegraphics[height=3in,width=4.2in]{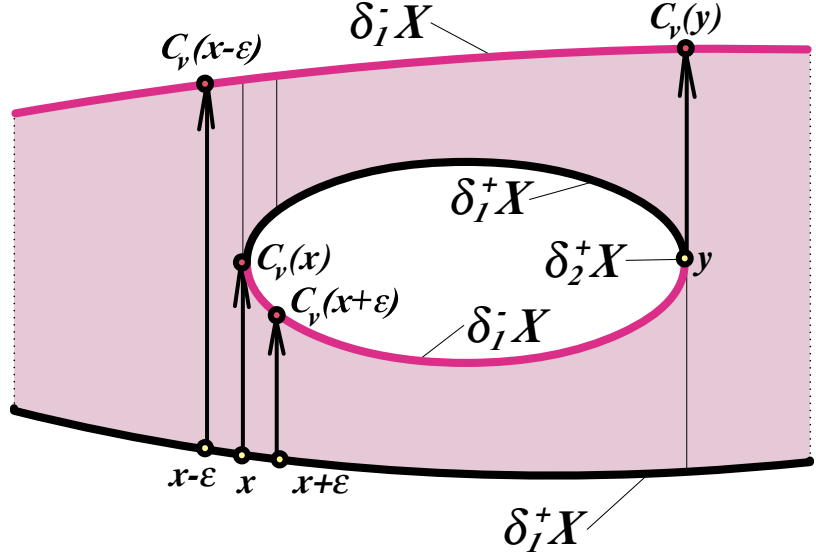}}
\bigskip
\caption{\small{The vertical traversing vector field $v$ on a surface $X$ divides its boundary $\delta_1X$ into two loci, $\delta_1^+X$ and $\delta_1^-X$. The causality map $C_v: \delta_1^+X \to \delta_1^-X$ is shown by the vertical arrows. The figure emphasizes the discontinuous nature of the map $C_v$.}} 
\end{figure}

We will employ boundary generic or even traversally generic metrics $g$ of the gradient type in order to control well the local structure of the causality map $C_{v^g}$.
\smallskip

For each tangent vector $w \in T_xM|_{\d M}$, consider its orthogonal decomposition $w = n \oplus u$ with respect to the metric $g$, where $n$ is the exterior normal to $\d M$  and $u \in T_x(\d M)$.  We denote by $\tau_g(x, w)$ the point  $(x, -n \oplus u) \in  T_xM|_{\d M}$.

\begin{lemma}\label{lem9.4} The $g$-independent manifolds $\d_1^+(SM)(v^g)$ and $\d_1^-(SM)(v^g)$ are diffeomorphic via the orientation-reversing involution $\tau_g: \d_1(SM) \to \d_1(SM)$.
\end{lemma}

\begin{proof}  Examining the definition of the strata $\d^\pm_1(SM)(v^g)$ and the construction of $\tau_g$, we see that $\tau_g$ maps $\d_1^+(SM)(v^g)$ to $\d_1^-(SM)(v^g)$ by an orientation-reversing diffeomorphism.
\end{proof}

Given a compact Riemannian manifold $(M, g)$, we denote by $\mathcal F(v^g)$ the oriented $1$-dimensional foliation on $SM$, produced by the geodesic field $v^g$. 

Let $\psi^t: SM \to SM$ be the $v^g$-flow generated local diffeomorphism; for each $x \in SM$, the image $\psi^t(x)$ is well-defined only for some values of $t$, a time interval in $\R$.
  
\begin{definition}\label{def9.4} Given two compact Riemannian $n$-manifolds, $(M_1, g_1)$ and $(M_2, g_2)$, consider the  geodesic fields $v^{g_1}$ on $SM_1$ and $v^{g_2}$ on $SM_2$.
\begin{itemize}
\item We say that the metrics $g_1$ and $g_2$ are \emph{geodesic flow topologically strongly conjugate} if there is a homeomorphism $\Phi: SM_1 \to SM_2$ such that $(\Phi \circ \psi^t_1)(x) = (\psi^t_2 \circ \Phi)(x)$ for all $x \in SM_1$ and all moments $t \in \R$ for which $\psi^t_1(x)$ is well-defined\footnote{This implies that $\Phi$ preserves the lengths of the corresponding trajectories in the Sasaki metrics, and thus the lengths of the corresponding geodesic curves in $M_1$ and $M_2$ are equal.}. The restriction of $\Phi$ to each $v^{g_1}$-trajectory is required to be an orientation-preserving diffeomorphism.\smallskip

\item We say that the metrics $g_1$ and $g_2$ are \emph{geodesic flow topologically conjugate} if there is a homeomorphism $\Phi: SM_1 \to SM_2$ such that it maps each leaf of $\mathcal F(v^{g_1})$ to a leaf of $\mathcal F(v^{g_2})$, the map $\Phi$ on every leaf being an orientation-preserving diffeomorphism. 

\hfill $\diamondsuit$
\end{itemize} 
\end{definition}

Both notions of conjugacy come in different flavors by requiring that $g_i, \Phi^\d$, and $\Phi$ belong to various classes of $C^k$-smooth objects, where $k = 0, 1, 2, \dots , \infty$. For example, we may consider  $\Phi^\d, (\Phi^\d)^{-1}$ of the class $C^\infty$, while $\Phi, (\Phi)^{-1}$ may be just homeomorphisms (of the class $C^0$).  
\smallskip

For complete (in particular, closed) manifolds, the investigation of geodesic flow topologically conjugate metrics\footnote{These investigations employ a notion of geodesic conjugacy similar to the one in the first bullet of Definition \ref{def9.4}.}  led to a variety of strong results \cite{BCG}, \cite{Cr}, \cite{Cr1}, \cite{CK} \cite{CEK}, \cite{Mat}, \cite{SUV}, \cite{SU4}. Let us describe their spirit: under certain conditions, imposed on metrics a priori, the geodesic flow topological conjugacy implies an isometry of the underlying metrics! In some cases, to establish the isometry, one needs to know also the lengths of geodesic lines (the, so called, the \emph{lens data}). In particular, in \cite{CEK}, the following  result has been established. 

\begin{theorem}{\bf (Croke, Eberlein, Kleiner)}
Let $(N_1, g_1)$ and $(N_2, g_2)$ be two closed Riemannian manifolds, $\dim N_1 = \dim N_2 \geq 3$. Assume that both manifolds have nonpositive sectional curvatures and that one of them has rank $k \geq  2$. 
Denote by $\{\psi^t_i\}_{t \in \R}$, $i = 1, 2$, the $g_i$-induced geodesic flow on $SN_i$.

If there is a homeomorphism $\Phi: SN_1 \to SN_2$ such that $\Phi \circ \psi^t_1 = \psi^t_2 \circ \Phi$ for all $t \in \R$, then $g_1$ and $g_2$ are isometric. \hfill $\diamondsuit$
\end{theorem}

Besson, Courtois, and Gallot proved the following equally striking theorem (\cite{BCG}, Theorem 1.3).  

\begin{theorem}{\bf (Besson, Courtois, and Gallot)} Let $(N, g)$ be a closed locally symmetric manifold of a negative sectional curvature and of dimension $\geq 3$.

Then any Riemannian manifold $(N_1, g_1)$ whose geodesic flow is $C^1$-conjugate\footnote{that is, the geodesic flows conjugating map $\Phi: SN_1 \to SN_2$ is a $C^1$-diffeomorphism.} to that of $(N, g)$ is isometric to $(N, g)$. \hfill $\diamondsuit$
\end{theorem}

Motivated by the spirit of these theorems (as far as we understand, they do not apply directly to the geodesic flows on manifolds with boundary), we move towards linking the geodesic flow topologically conjugate Riemannian manifolds with the problem of \emph{inverse geodesic scattering}.  In its more daring formulation, the problem of inverse geodesic scattering asks to reconstruct the metric $g$ on $M$ from the scattering map $C_{v^g}: \d_1^+(SM) \to \d_1^-(SM)$, the reconstruction is thought up to the natural action of $\mathsf{Diff}(M, \d M)$, the group of diffeomorphisms that are the identity maps on $\d M$, on the Riemannian metrics (see Remark 3.2). \smallskip

In the definition below, we assume that a given smooth compact Riemannian manifold $M$ is embedded properly into a larger open manifold $\hat M$, and the metric $g$ on $M$ is extended smoothly to a metric $\hat g$ on $\hat M$. However, the properties, described in the definition, do not depend on a particular extension $(\hat M, \hat g) \supset (M, g)$. 

\begin{definition}\label{def9.4a} We say that a boundary generic metric $g$ of the gradient type on a compact manifold $M$ has the property $\mathsf A$, if \emph{one} of the following statements holds: 
\begin{itemize}
\item Any geodesic curve $\g \subset M$, but a singleton, has at least one point of transversal intersection with the boundary $\d M$. If the geodesic $\hat\g \subset \hat M$ is such that $\hat\g \cap M$ is a singleton $m$, then $\hat\g$ is 
quadratically tangent to $\d M$ at $m$. \smallskip

\item Any geodesic curve $\hat\g \subset \hat M$ that is tangent to $\d M$, is quadratically tangent  to it. \hfill $\diamondsuit$
\end{itemize}
\end{definition}

\noindent {\bf Remark 3.1.}
We will see later that the first bullet in property $\mathsf A$ implies that any trajectory of the geodesic flow $v^g$ on $SM$, but a singleton, has at least one point of transversal intersection with the boundary $\d(SM)$. The trajectories-singletons are quadratically tangent to $\d(SM)$ in $S\hat M$. In terms of \cite{K2} and \cite{K4}, the combinatorial tangency types of such trajectories do not belong to the closed poset $(33)_\succeq \cup (4)_\succeq \subset \Omega^\bullet$.

Similarly, the second bullet in property $\mathsf A$ implies that, if a $v^{\hat g}$-trajectory is tangent to $\d(SM)$ at a point, then it is quadratically tangent there.  In other words, the combinatorial tangency types of such trajectories do not belong to the closed poset $(3)_\succeq  \subset \Omega^\bullet$. 
\smallskip

The property $\mathsf A$ reflects the shortcomings of our proof of the \emph{smooth} version of The Holography Theorem 3.1 from \cite{K4}. We suspect that it is a superfluous assumption, and the conjugating homeomorphism in The Holography Theorem is actually a diffeomorphism. Therefore, property $\mathsf A$ is likely a superfluous constraint, when the conjugating homeomorphism $\Phi: SM_1 \to SM_2$ is desired to be a diffeomorphism. Regrettably,  we must include property $\mathsf A$ as a hypotheses in some theorems to follow. 
\hfill $\diamondsuit$
\smallskip

\noindent{\bf Remark 3.2.}
Note that the inverse scattering problems on Riemannian manifolds $(M, g)$ have an unavoidable intrinsic ambiguity. It arises from the action of the $M$-diffeomorphisms that are the identity on $\d M$ and whose differentials are the identity on $TM |_{\d M}$. Let us denote by $\mathsf{Diff}(M, \d M)$ the group of such diffeomorphisms. Each diffeomorphism $\phi \in \mathsf{Diff}(M, \d M)$ acts naturally on $g$, producing a new metric $\phi^\ast(g)$. Since such $\phi$ maps geodesics in $g$ to geodesics in $\phi^\ast(g)$,  $(M, g)$ and $(M, \phi^\ast(g))$ share the same the scattering map. Therefore, for a given random $M$ and $C_{v^g}$, the best we can hope for is to reconstruct $g$ up to the $\mathsf{Diff}(M, \d M)$-action. 

For any $\phi \in \mathsf{Diff}(M, \d M)$, the geodesic flows $v^g$ and $v^{\phi^\ast(g)}$ on $TM \setminus M$ are strongly conjugated with the help of the diffeomorphism $D\phi$.
\hfill $\diamondsuit$
\smallskip

Our main result, Theorem \ref{th9.2} below, claims: 

\emph{For smooth boundary generic Riemannian metrics of the gradient type, the inverse geodesic scattering problem is topologically rigid, up to the geodesic flow topologically conjugate equivalence\footnote{see the second bullet of Definition \ref{def9.4}} among metrics}.

\begin{theorem}\label{th9.2}{\bf (the topological rigidity of the geodesic flow for the inverse scattering problem)}\smallskip

Let $(M_1, g_1)$ and  $(M_2, g_2)$ be two smooth compact connected Riemannian $n$-manifolds with boundaries.  Let the metrics $g_1$, $g_2$ be geodesically boundary generic, and let $g_2$ be of the gradient type.
 
Assume that the scattering maps $$C_{v^{g_1}}: \d_1^+(SM_1) \to \d_1^-(SM_1)\;\; \text{and} \;\; C_{v^{g_2}}: \d_1^+(SM_2) \to \d_1^-(SM_2)$$ are conjugate by a smooth diffeomorphism $\Phi^\d: \d_1(SM_1) \to \d_1(SM_2)$.  \smallskip

Then $g_1$ is also of the gradient type. Moreover, the metrics $g_1$ and $g_2$ are geodesic flow \emph{topologically} conjugate.\smallskip

If  the metric $g_2$ possess property $\mathsf A$ from Definition \ref{def9.4a}, then so does $g_1$, and the conjugating homeomorphism $\Phi: SM_1 \to SM_2$ is a diffeomorphism.   
\end{theorem}

\begin{proof} By the hypotheses, the geodesic field  $v^{g_2}$ is of the gradient type (equivalently, traversing) and boundary generic. If the causality map $C_{v^{g_1}}$ is not well-defined for some $w \in \d_1^+(SM_1)$, then by the property of $\Phi^\d$, the map $C_{v^{g_2}}$ is not well-defined at $\Phi^\d(w)$, a contradiction to the property of $v^{g_2}$ being traversing. Therefore the $v^{g_1}$-trajectory that connects $w$ with $C_{v^{g_1}}(w)$ is a closed segment or a singleton for any $w \in \d_1^+(SM_1)$; in other words, $v^{g_1}$ is a traversing field. By Lemma 4.1 from  \cite{K1}, any traversing field is of the gradient type. So $v^{g_1}$ is of the gradient type. Thus both causality maps $C_{v^{g_1}}$ and $C_{v^{g_2}}$ are well-defined.  

By applying the Holography Theorem 3.1 and Corollary 3.3 from \cite{K4}, we conclude that the two geodesic flows are topologically conjugate (in the sense of Definition \ref{def9.4}) by a homeomorphism $\Phi: SM_1 \to SM_2$. When property $\mathsf A$ is valid for one of the two manifolds, by the proof of the Holography Theorem 3.1 from  \cite{K4}, it is valid for the other one. So in this case, by the Holography Theorem, $\Phi$ is a diffeomorphism. 
\end{proof}

\begin{corollary}\label{cor9.4} Assume that a compact connected and smooth manifold $M$ with boundary admits a boundary generic  Riemannian metric $g$ of the gradient type. Let $\mathcal F(v^g)$ denote the oriented $1$-dimensional foliation, induced by the geodesic flow. \smallskip

Then the scattering map $C_{v^g}: \d_1^+(SM) \to \d_1^-(SM)$ allows for a reconstruction of the pair $(SM, \mathcal F(v^g))$, up to a homeomorphism of $SM$ which is the identity on $\d(SM)$ and an orientation-preserving diffeomorphism on each $v^g$-trajectory. \smallskip

If  $g$ possess property $\mathsf A$ from Definition \ref{def9.4a}, then it is possible to reconstruct the pair $(SM, \mathcal F(v^g))$, up to a diffeomorphism of $SM$ which is the identity on $\d(SM)$.
\hfill $\diamondsuit$
\end{corollary}

\noindent {\bf Example 3.1.} 
Consider a shell $M$, produced by removing a strictly convex domain from the interior of a strictly convex domain in the Euclidean space $\mathsf E$ (so topologically $M$ is a shell). Any geodesic line $\g$ in $M$ that is tangent to the boundary of the interior convex domain is transversal to the boundary of the exterior convex domain. The intersection $\g \cap \d M$ of any geodesic line $\g \subset \mathsf E$ that is tangent to the boundary of the exterior convex domain is a singleton. For such a pair $(M, g_\mathsf E)$, the geodesic field $v^g$ on $SM$ is traversally generic; in particular, it is boundary generic and of the gradient type. Thus Theorem \ref{th9.3} is applicable to $(M, g_\mathsf E)$, as well as to any metric $g$ in $M$, sufficiently close to $g_\mathsf E$. \hfill $\diamondsuit$
\smallskip

Under the hypotheses of Theorem \ref{th9.2}, the scattering maps faithfully distinguish between manifolds whose spherical tangent bundles have distinct topological types:

\begin{corollary}\label{cor9.4a} Let $(M_1, g_1)$ and  $(M_2, g_2)$ be two smooth compact connected Riemannian $n$-manifolds with boundaries. Let the metrics $g_1$, $g_2$ be of the gradient type and geodesically boundary generic. 
Assume that the boundaries $\d M_1$ and $\d M_2$ are diffeomorphic\footnote{This implies that $\d(SM_1)$ and $\d(SM_2)$ are diffeomorphic.}, but the spaces $SM_1$ and $SM_2$ are not homeomorphic. 

Then no diffeomorphism $\Phi^\d: \d(SM_1) \to \d(SM_2)$ conjugates the two scattering maps $C_{v^{g_1}}$ and  $C_{v^{g_2}}$.  \hfill $\diamondsuit$
\end{corollary}

\noindent{\bf Example 3.2.}
Recall that the classical knots are determined by the topological types of their complements in $S^3$ or $\R^3$ (the Gordon-Luecke Theorem \cite{GL}), while the links are not; in fact, there are infinitely many distinct links whose complements are all homeomorphic to the complement of the Whitehead link. Let us examine Corollary \ref{cor9.4a}, while keeping these facts in mind.

Let $L_1 \subset D^3_1 \subset N_1$ and $L_2 \subset D^3_2 \subset N_2$ be two links in closed $3$-folds $(N_1, g_1)$ and $(N_2, g_2)$, respectively. We denote by $U_1$ and $U_2$ the regular tubular neighborhoods of $L_1 \subset D^3_1$ and $L_2 \subset D^3_2$. Put $M_1 = N_1 \setminus \textup{int}(U_1)$, $M_2 = N_2 \setminus \textup{int}(U_2)$. We assume that $L_1$ and $L_2$ have the same number of components; thus there exists a diffeomorphism $\psi: U_1 \to U_2$

If $g_1|_{M_1}$ and  $g_2|_{M_2}$ are boundary generic and of the gradient type, and $SM_1$ and $SM_2$ are \emph{not} homeomorphic, then the scattering maps $C_{v^{g_1}}$ and $C_{v^{g_2}}$ are not conjugated via any diffeomorphism $\Phi^\d: \d(SU_1) \to \d(SU_2)$. 

So the scattering maps \emph{distinguish} between the links $L_1$ and $L_2$ with non-homeomorphic \emph{stabilized} complements $SM_1$ and $SM_2$.  When $N_i = S^3$, then $SM_i$ is homeomorphic to $M_i \times S^2$.  

In particular, consider the sphere $S^3$, equipped  with the standard metric $g_S$. Let $D^3_\bullet \subset S^3$ be a smooth ball that properly contains a hemisphere $S^3_+$. Then any geodesic in $S^3$ hits $D^3_\bullet$.  Take a pair of links $L_i \subset S^3$, $i=1, 2$, such that their tubular neighborhoods $U_i$ are formed by attaching solid $1$-handles to $D^3_\bullet$. Then any geodesic on $S^3$ hits  $D^3_\bullet$ and thus each $U_i$. Therefore the metric $g_S$ on $M_i = S^3 \setminus U_i$ is of the gradient type. Assuming that the spherical metric is boundary generic on each $M_i$ (conjecturally, this property may be achieved by a smooth perturbation of $\d U_i$), we conclude that if the two scattering maps on $M_1$ and $M_2$ are smoothly conjugated, then $M_1 \times S^2$ must be homeomorphic to $M_2 \times S^2$.  

\hfill $\diamondsuit$
\smallskip

Note that a Riemannian manifold  $(M, g)$ produces a $(2n-2)$-bundle $\xi^g$ over $SM$. It is a subbundle of $T(SM)$, transversal  to the $1$-foliation $\mathcal F(v^g)$. The isomorphism class of $\xi^g$ is well-defined by $v^g$. In fact, for a gradient-like $g$, the distribution $\xi^g$ is integrable: the tangent bundle to the $(2n-2)$-foliation $\{F^{-1}(c)\}_{c \in \R}$, where $dF(v^g) > 0$, delivers $\xi^g$. 

On the other hand, employing any $v^g$-invariant Riemannian metric $\mathsf g$ on $SM \subset TM$, we may consider the $(2n-2)$-distribution $(v^g)^\perp$, orthogonal in $\mathsf g$ to $v^g$. That distribution is locally invariant under the $v^g$-flow, and thus generates a $(2n-2)$-bundle $\tau^g$ over the space of un-parametrized geodesics $\mathcal T(v^g)$, such that $\Gamma^\ast(\tau^g) = \xi^g$, where $\Gamma: SM \to \mathcal T(v^g)$ is the obvious map. 
\smallskip

\begin{corollary}\label{cor9.1a} Under the notations and hypotheses of Theorem \ref{th9.2}, the geodesic flows conjugating homeomorphism $\Phi: SM_1 \to SM_2$ (which extends $\Phi^\d$) induces a bundle isomorphism  $\Phi^\ast: \xi^{g_2} \to \xi^{g_1}$. 

Similar,  the $\Phi^\d$-generated natural map $\Phi^\mathcal T : \mathcal T(v^{g_1}) \to \mathcal T(v^{g_2})$ of the spaces of geodesics induces the ``tangent" bundle isomorphism $(\Phi^\mathcal T)^\ast: \tau^{g_2} \to \tau^{g_1}$. 
\end{corollary}

\begin{proof} By the proof of Holography Theorem 3.1 from \cite{K4}, the homeomorphism $\Phi$ has the property $\Phi^\ast(F_2) = F_1$, where the function $F_i: SM_i \to \R$ is such that $dF_i(v^{g_i}) > 0$ ($i = 1, 2$).

Since both geodesic flows are traversing, $\mathcal T(v^{g_i}) = \d(SM_i)/C_{v^{g_i}}$, where $\d(SM_i)/C_{v^{g_i}}$ denotes the quotient space by the partially-defined scattering map $C_{v^{g_i}}$. Since $\Phi^\d \circ C_{v^{g_1}} =  C_{v^{g_2}} \circ \Phi^\d$, the diffeomorphism $\Phi^\d: \d(SM_1) \to \d(SM_2)$ generates the homeomorphism $\Phi^\mathcal T : \mathcal T(v^{g_1}) \to \mathcal T(v^{g_2})$. Moreover, by the proof of Theorem \ref{th9.2}, the homeomorphism $\Phi^\mathcal T$ has the property $\Gamma_2 \circ \Phi = \Phi^\mathcal T \circ \Gamma_1$. Therefore, using that $\Gamma_i^\ast(\tau_i) = \xi_i$, we conclude that $(\Phi^\mathcal T)^\ast: \tau^{g_2} \to \tau^{g_1}$ is a bundle isomorphism.  
\end{proof}

\noindent {\bf Remark 3.3.}
If a pair $(M, g)$ is such that the geodesic field $v^g$ is traversing on $SM$, then the trajectory space $\mathcal T(v^g)$---\emph{the space of un-parametrized geodesics} on $M$--- is a separable compact space. It is given the quotient topology via the obvious map $\pi: SM \to \mathcal T(v^g)$. For a geodesically boundary generic metric $g$ of the gradient type, the map $\pi |: \d(SM) \to \mathcal T(v^g)$ has finite fibers. 

The space of geodesics $\mathcal T(v^g)$, although not a manifold in general, for traversing geodesic flows, inherits a surrogate smooth structure from $SM$ (as in Definitions 2.2 and 2.3 from \cite{K4}). This structure manifests itself in the form of the algebra $C^\infty(\mathcal T(v^g))$ of smooth functions  $f \in C^\infty(SM)$ such that the directional derivatives $\mathcal L_{v^g}(f) = 0$. 

The bundle $\tau^g$ may be viewed as a surrogate tangent bundle of the stratified singular space $\mathcal T(v^g)$, since its restriction to the open and dense set $\mathcal T(v^g, \omega_0)$ of trajectories of the combinatorial tangency type $\omega_0 = (1,1)$ may be identified with the ``honest" tangent bundle of the open manifold $\mathcal T(v^g, \omega_0)$ (\cite{K3}, \cite{K4}).  
\hfill $\diamondsuit$
\smallskip

Recall that for any traversing boundary generic field $v$ on a compact smooth manifold $X$, each $v$-trajectory $\g$ produces an ordered finite sequence $\omega = (\omega_1, \omega_2, \dots)$ of positive integral multiplicities, associated with the finite ordered set $\g \cap \d X$. These sequences $\omega$ can be  organized in an \emph{universal poset} $\Omega^\bullet$ \cite{K3}. \smallskip

In view of Remark 3.2 and Corollary 3.2 from \cite{K4}, Theorem \ref{th9.2} has an instant, but philosophically important implication:

\begin{theorem}\label{th9.3} Assume that $M$ admits a geodesically boundary generic Riemannian metric $g$ of the gradient type. Then the scattering map $C_{v^g}: \d_1^+(SM) \to \d_1^-(SM)$  allows for a reconstruction of: 
\begin{itemize}

\item the $\Omega^\bullet$-stratified  topological type of the space $\mathcal T(v^g)$  of un-parametrized geodesics on $M$,

\item  the isomorphism class of the ``tangent" $(2n-2)$-bundle $\tau^g$ of the space $\mathcal T(v^g)$, 

\item  the  smooth topological type of the space $\mathcal T(v^g)$, provided that the property $\mathsf A$ from Definition \ref{def9.4a} is valid. 
\end{itemize}
\end{theorem}

\begin{proof} By definition, $\mathcal T(v^g)$ is the quotient space of $SM$ under the obvious map $\Gamma$ that takes each $v^g$-trajectory to a singleton. Now just apply Theorem \ref{th9.2} and Corollary \ref{cor9.4} to reconstruct the map $\Gamma$. By  Corollary 3.2 from \cite{K4}, $C_{v^g}$ allows for a reconstruction of the $\Omega^\bullet$-stratified  topological type of the space $\mathcal T(v^g)$.\smallskip

 By combining Corollary \ref{cor9.1a} with Theorem \ref{th9.3}, we get the second claim.\smallskip

When property $\mathsf A$ is valid, this application will lead to a reconstruction, up to an algebra isomorphism, of the algebra $C^\infty(\mathcal T(v^g), \R)$ of smooth functions on $\mathcal T(v^g)$, the kernel of the directional derivative operator $\mathcal L_{v^g}: C^\infty(SM, \R) \to C^\infty(SM, \R)$.
\end{proof}
\smallskip

\begin{question}\label{q9.2} 
\emph{Given a  boundary generic metric $g$ on a compact smooth $n$-manifold $M$, how to describe the Morse strata $\{\d_j^\pm SM(v^g)\}_{1 \leq j \leq 2n-1}$ of $\d(SM)$ in terms of} the local Riemannian  geometry\footnote{like the curvature tensor of $M$ and the normal curvature of $\d M$}  \emph{of the pair  $(M, \d M)$}? 
\hfill $\diamondsuit$
\end{question}
\smallskip

The next two lemmas represent a step towards answering Question 3.1. The language,  in which the answer is given, is reminiscent of the language in \cite{Zh}.

As before,  we assume that a given smooth compact Riemannian manifold $M$ is embedded properly into a larger open manifold $\hat M$, and the metric $g$ on $M$ is extended smoothly to a metric $\hat g$ on $\hat M$.

\begin{lemma}\label{lem9.5} For a boundary generic metric $g$ on $M$, the stratum $\d_jSM(v^g)$ consists of pairs $(x, w) \in SM|_{\d M}$ such that the germ of the geodesic curve $\g \subset \hat M$ through $x$ in the direction of $w$ is tangent to $\d M$ with the multiplicity $m(x) = j -1$. Moreover, the stratum $\d_j^\pm SM(v^g)$ also has a similar description\footnote{Its exact formulation is described in the proof.} in terms of the germ of $\g \subset M$. 
\end{lemma}

\begin{proof} Take a smooth function $z: \hat M \to \R_+$ with the properties: 

 {\bf (1)} $0$ is a regular value of $z$, 
 
 {\bf (2)} $z^{-1}(0) = \d M$, 
  
 {\bf (3)} $z^{-1}((-\infty, 0]) = M$ \footnote{We may use the distance functions $-d_g(\sim, \d M)$ in $M$ and $d_g(\sim, \d M)$ in $\hat M \setminus M$ for the role of $z$.} 
 
Let $\tilde z: S\hat M \to \R_+$ be the composition of the projection $\pi: S\hat M \to \hat M$ with $z$. Evidently, $\tilde z: S\hat M \to \R$ has the same three properties with respect to $\d(SM)$ as $z: \hat M \to \R$ has with respect to $\d M$.

For any pair $(x, w) \in \d(SM)$, consider the germ of the geodesic curve $\g \subset \hat M$ through $x \in \d M$ in the direction of $w \in T_xM$ and its lift $\tilde\g \subset S\hat M$, the geodesic $v^{\hat g}$-flow curve through $(x, w) \in \d(SM)$. Then $\pi: \tilde\g \to \g$ locally is an orientation preserving  diffeomorphism of curves.  Therefore the $k$-jet $jet^k(\tilde z|_{\tilde\g}) = 0$ if and only if $k$-jet $jet^k(z|_\g) =0$. 

By Lemmas 3.1 and 3.4 from \cite{K2}, $(x, w) \in \d_j(SM)(v^g)$ if and only if $jet^{j-1}(\tilde z|_{\tilde\g}) = 0$; similarly, $(x, v) \in \d_j^+(SM)(v^g)$ if, in addition, $\frac{\d^j}{(\d \tilde t)^j}(\tilde z|_{\tilde\g})(x,v) \geq 0$ (here $\tilde t$ is the natural parameter along $\tilde\g$). \smallskip

Thanks to the orientation-preserving diffeomorphism $\pi: \tilde\g \to \g$, these properties of $\tilde z|_{\tilde\g}$ are equivalent to the similar properties $jet^{j-1}( z|_{\g}) = 0$ and $\frac{\d^j}{(\d t)^j}(z|_{\g})(x) \geq 0$ of $z|_\g$. The latter ones describe the multiplicity $m(x) =_{\mathsf{def}} j - 1$ of tangency of $\g$  to $\d M$ at $x$. 
\end{proof}

\begin{corollary}\label{cor3.2} Let $M$ be a smooth compact Riemannian manifold of dimension $n$. For a boundary generic metric $g$ on $M$, the order of tangency of any geodesic curve to $\d M$ does not exceed $2n-1$. 
\end{corollary}

\begin{proof}  For any boundary generic vector field $w$ on $SM$, the locus $\d_{\dim(SM)+1}(SM)(w) = \emptyset$. In particular, since $\dim(SM) = 2n-1$, we get $\d_{2n}(SM)(v^g) = \emptyset$. 
\end{proof}

Let us rephrase the previous arguments in the proof of Lemma \ref{lem9.5} in terms of the exponential maps. 

For each point $x \in \d M$, consider the $\hat g$-induced exponential map $\exp_x^{\hat g}: T_xM \to \hat M$ and its locally-defined inverse $\ln_x^{\hat g}: \hat M \to T_xM$. We may assume that $\exp_x^{\hat g}$ is well-defined for all sufficiently small $w \in T_xM$ and $\ln_x^{\hat g}$ for all $y \in \hat M$ that are close to $x$. The local maps $\{\exp_x^{\hat g}\}_{x \in \d M}$ can be organized into a single map $\exp^{\hat g}: TM|_{\d M} \to \hat M$, well-defined in a neighborhood of the zero section $\s: \d M \to TM|_{\d M}$.

We denote by $H_x: T_xM \to \R$ the pull-back of the function $z: \hat M \to \R$ under the map $\exp_x^{\hat g}$. Similarly, we denote by $H: TM|_{\d M} \to \R$ the pull-back of $z: \hat M \to \R$ under the map $\exp^{\hat g}$. 

Pick $w \in T_xM$, where $x \in \d M$. Then the order of tangency of the hypersuface $\d M = \{z= 0\}$ with the geodesic curve $\exp_x^{\hat g}(tw) \subset \hat M$ is equal to the order of tangency of the hypersurface $\{H_x(w) = 0\} \subset T_xM$ with the line $\{tw\}_{t \in \R}$. Therefore the property $(x, w) \in \d_j(SM)(v^g)$ is equivalent to the property $jet^{j-1}_{t=0}(H_x(tw)) = 0$, where $jet^{j-1}_{t=0}$ denotes the $(j-1)$-jet of the $t$-function $H_x(tw)$ at $t = 0$, and $w$ is a unitary tangent vector at $x$. 
\smallskip

Let  $\pi: TM \to M$ denote the tangent bundle, and $\s: M \to TM$ its zero section. Let $\pi^\ast: T^\ast M \to M$ denote the cotangent bundle. 

\begin{definition}\label{def9.6}
We say that two smooth functions $f, h \in C^\infty(TM, \R)$ share the same \emph{vertical} $k$-\emph{jet field} and write $f \equiv_\pi^k h$, if $jet^k_{0}(f|_{\pi^{-1}(x)}) =  jet^k_{0}(h|_{\pi^{-1}(x)})$ for all $x \in M$. Here $jet^k_{0}(F|_{\pi^{-1}(x)})$ stands for the $k$-jet at the point $\s(x) = 0 \in T_xM$ of a smooth function $F: TM \to \R$, being restricted to the $\pi$-fiber $T_xM$. 

We denote by $\mathcal J^k(\pi)$\footnote{not to be confused with the standard notation ``$J^k(\pi)$" that refers to the equivalence classes of local sections of the bundle $\pi$, a ``horizontal" construction!} the quotient space of $C^\infty(TX, \R)$ by the equivalence relation \hfill\break `` $\equiv_\pi^k\,$". 
\hfill $\diamondsuit$
\end{definition}

Consider the $k$-jet space $Jet^k(\R^n, \R)$, jets being formed at the origin $0 \in \R^n$.  We may interpret $Jet^k(\R^n, \R)$ as the space of real polynomials of degree $k$ at most in $n$ variables. 

Let $\mathcal S^j((\R^n)^\ast)$ denotes the space of homogeneous polynomials of degree $j$ in $n$ variables.

Any jet $\tau \in Jet^k(\R^n, \R)$ has a unique representation as a sum $\tau_0 +\dots + \tau_k$, where the ``homogeneous" summands $\tau_j \in \mathcal S^j((\R^n)^\ast)$. So the $k$-jets can be decomposed:  $Jet^k(\R^n, \R) \approx \bigoplus_{j= 0}^k \mathcal S^j((\R^n)^\ast)$. 
\smallskip

The group $\mathsf{GL}(n)$ acts on $\R^n$, and thus on $Jet^k(\R^n, \R)$. Note that 
each subspace $\mathcal S^j((\R^n)^\ast)$ is preserved by this action. \smallskip

Let $\Pi: PM \to M$ be the principle $\mathsf{GL}(n)$-bundle, associated with the bundle $\pi: TM \to M$. We can form the bundle $Jet^k_{vert}(\pi)  \to M$, associated with the bundle $\Pi: PM \to M$. Its total space is $PM \times_{\mathsf{GL}(n)} Jet^k(\R^n, \R)$ and its fiber is $Jet^k(\R^n, \R)$. 

Similarly, we may construct a bundle $\mathcal S^j(\pi^\ast) = P^\ast M \times_{\mathsf{GL}(n)} \mathcal S^j((\R^n)^\ast)$ over $M$, associated with the cotangent bundle $\pi^\ast: T^\ast M \to M$.  Employing the $\mathsf{GL}(n)$-equivariant isomorphism $\b: Jet^k(\R^n, \R) \approx \bigoplus_{j=0}^{k} \mathcal S^j((\R^n)^\ast)$, we get a bundle isomorphism $B: Jet^k_{vert}(\pi) \approx \bigoplus_{j=0}^{k} \mathcal S^j(\pi^\ast)$.
\smallskip

Given a smooth bundle $\eta: E(\eta) \to M$, we denote by $\Gamma(\eta)$ the space of its smooth sections over $M$.  \smallskip

By the definition of the space $\mathcal Jet^k(\pi)$, there is the obvious injection $\mathcal E: \mathcal Jet^k(\pi) \longrightarrow \Gamma(J^k_{vert}(\pi))$. We use it to form the composite map
$$\mathsf T^k:\, C^\infty(TM, \R) \stackrel{\equiv_\pi^k}{\longrightarrow} \mathcal Jet^k(\pi) \stackrel{\mathcal E}{\longrightarrow}  \Gamma(Jet^k_{vert}(\pi)) \approx \bigoplus_{j=0}^{k} \Gamma(\mathcal S^j(\pi^\ast)).$$
Therefore, for a fixed $k$, each function $f \in  C^\infty(TM, \R)$ gives rise to a sequence of sections $\{T^k_j(f) \in \Gamma(\mathcal S^j(\pi^\ast))\}_{j \in [0, k]}$. 

In particular, $T^k_0(f) = f|_{\s(M)},\, T^k_1(f) = d_\pi f|_{\s(M)}$, where $d_\pi$ stands for the differential along the $\pi$-fibers. \smallskip

We call the sum $\mathsf T^k(f) = \sum_{j=0}^k T^k_j(f)$ the \emph{vertical Taylor polynomial of} $f$ of degree $k$.
\smallskip

Each of the sections $T^k_j(f)$ can be evaluated at any point $(x, w) \in TM$ by forming the symmetric contravariant tensor $T^k_j(f)(x) \in \mathcal S^j(T^\ast_xM)$ and evaluating it at the polyvector $w^{\otimes_S j} =_{\mathsf{def}} \underbrace{w \otimes_S w \dots \otimes_S w}_{j},$ where ``$\otimes_S$" stands for the symmetrized tensor product of vectors. We denote by $\tau^k_j(f)(x, w)$ the result of this evaluation.
\smallskip

Now let us choose $k = 2n - 1 = \dim(SM)$. Consider the vertical Taylor expansion $T^k(H) = \sum_{j=0}^k T_j^k(H)$ of the function $$H(x, w) =_{\mathsf{def}} z(\exp^{\hat g}_x(w)),$$ where $x \in \hat M$ is in the vicinity of $\d M$ and $T_j^k(H)(x, w^\ast)$ is a homogeneous polynomial in $w^\ast \in T^\ast_x\hat M$ of degree $j$. As before, we denote by $\tau_j^k(x, w) =_{\mathsf{def}} \tau_j^k(H)(x, w)$ the result of evaluation of the symmetric differential form $T_j^k(H)$ at the vector $w^{\otimes_S j} \in \mathcal S^j(T_x\hat M)$. Note that $\tau^k_0(x, w) = z(x)$.

Using the homogeneity of $\tau^k_j(x, w)$ in $w$ and in light of the arguments that precede Definition \ref{def9.6}, we conclude that $(x, w) \in \d_j(SM)(v^g)$ is equivalent to the requirement $\{\tau^k_l(x,w) = 0\}_{0 < l < j}$, where $w$ is a unit tangent vector and $k \geq j$. Moreover, $(x, w)$ belongs to the pure stratum $\d_j^\circ(SM)(v^g)$ if and only if $\{\tau^k_l(x, w) = 0\}_{0 < l < j}$ and $\tau^k_j(x, w) \neq 0$. 
\smallskip

Although the function $H$ depends on the choice of the auxiliary function $z$, the solution set $\{\tau^k_l(x,w) = 0\}_{0 < l < j}$ does not. Indeed, replacing the function $z$ with the new function $z\cdot q$, where $q|_{\d M} > 0$, and using the multiplicative properties of the local Taylor expansions, leads to a new system of constraints $\{\tilde \tau^k_l(x,w) = 0\}_{0 < l < j}$ which exhibit a ``solvable triangular pattern", as compared with the original system $\{\tau^k_l(x,w) = 0\}_{0 < l < j}$. \smallskip 

\noindent {\bf Remark 3.1.}
For each $x \in \d M$, the degree $l$ homogeneous in $w$ equations $\{\tau^k_l(x,w) = 0\}_{l < j}$ define a $z$-independent real algebraic subvariety  $\mathcal V_{< j}(x)$ of the unit sphere $S_x \subset T_xM$.  The pure stratum $\d_j^\circ(SM)(v^g)$ gives rise to the semialgebraic set $\mathcal V^\circ_{j}(x) =_{\mathsf{def}} \mathcal V_{< j}(x) \setminus \mathcal V_{< j+1}(x)$. 
The subvariety $\mathcal V_{< j+1}(x)$ separates  $\mathcal V_{< j}(x)$ into two semialgebraic sets: $\mathcal V^+_{< j}(x)$, where $\tau^k_j(x, w) \geq 0$, and $\mathcal V^-_{< j}(x)$, where $\tau^k_j(x, w) \leq 0$. These loci reflect the \emph{generalized concavity and convexity} of $\d M$ (in the metric $g$) at $x \in \d M$ in the direction of a given unit vector $w \in S(T_xM)$.

Note that, for given $x$ and $j$, the sets $\mathcal V_{< j}(x), \mathcal V^+_{< j}(x), \mathcal V^-_{< j}(x)$ may be empty. So, with $\d M$ and $g$ being fixed, we may regard the requirements $\mathcal V_{< j}(x) \neq \emptyset,\, \mathcal V^+_{< j}(x) \neq \emptyset,\, \mathcal V^-_{< j}(x) \neq \emptyset$ as constraints, imposed on $x \in \d M$. \hfill $\diamondsuit$
\smallskip 

Consider now on the restrictions of the smooth functions $\tau_j^{2n-1}(H): T\hat M \to \R$ to the subspace $S\hat M$ in the vicinity of the boundary $\d(SM)$. Together, $\{\tau_j^{2n-1}(H)\}_j$ will generate a smooth map $\vec\tau^{2n-1}(H): S\hat M \to \R^{2n}$. 

Let $y_0,  \dots , y_{2n-1}$ denote the coordinates in $\R^{2n}$. Let $L_j$ be the subspace of $\R^{2n}$, defined by the equations $\{y_0 = 0, \dots , y_j = 0\}$. Consider the complete flag 
$$\mathsf F(2n) = \{\R^{2n} \supset L_0 \supset \dots \supset L_{2n-2} \supset L_{2n-1} = \{0\}\}.$$

Let us focus on the interaction of the map $\vec\tau^{2n-1}(H): S\hat M \to \R^{2n}$ with the flag $\mathsf F(2n)$. 

The considerations above imply that $\d_j(SM)(v^g) = (\vec\tau^{2n-1}(H))^{-1}(L_{j-1})$.
\smallskip

The definition of boundary generic geodesic field $v^{\hat g} \in T(S\hat M)$ with respect $\d(SM) = \{\tau^{2n-1}_0(H) = 0\}$ implies the linear independence of the differential $1$-forms $d\tau^{2n-1}_0(H), \dots , \break d\tau^{2n-1}_{j-1}(H)$ along the locus $$\d_j(SM)(v^g) = \{\tau^{2n-1}_0(H) = 0, \dots ,\tau^{2n-1}_{j-1}(H) = 0\}.$$ In turn, their linear independence is equivalent to the non-vanishing of $j$-form $$d\tau^{2n-1}_0(H) \wedge \dots \wedge d\tau^{2n-1}_{j-1}(H) \in \bigwedge^j T^\ast(SM) \neq 0$$ along the locus $\d_j(SM)(v^g)$.
\smallskip

In light of the considerations above, we get a ``more constructive" version of Lemma \ref{lem9.5}.

\begin{lemma}\label{lem9.5a} Let $n = \dim M$ and $k = 2n-1$. Consider an auxiliary function $z: \hat M \to \R$ with properties (1)-(3) from the proof of Lemma \ref{lem9.5} and the associated function
$$H : T\hat M|_{\d M} \stackrel{\exp^{\hat g}}{\longrightarrow} \hat M \stackrel{\hat z}{\longrightarrow} \R.$$

The metric $\hat g$ on $\hat M$  is boundary generic with respect to $\d M = z^{-1}(0)$ if and only if the map  $\vec\tau^k(H): S\hat M \to \R^{2n}$ is transversal to each space $L_j$ from the flag $\mathsf F(2n)$\footnote{By the properties of $z$, $\vec\tau^k(H)$ is transversal to $L_0$ and $\vec\tau^k(H)^{-1}(L_{0}) = \d(SM)$.}. Moreover, $\d_j(SM)(v^g) = \vec\tau(H))^{-1}(L_{j-1})$. \smallskip

In other words, $\hat g$ is boundary generic if and only if, for each $j \in [1, 2n]$, the differential $j$-form $$d\tau^{k}_0(H) \wedge d\tau^{k}_1(H) \wedge \dots \wedge d\tau^{k}_{j-1}(H) \in \bigwedge^j T^\ast(SM)$$ does not vanish along the locus $$\d_j(SM)(v^g) = \{\tau^{k}_0(H) = 0,\, \tau^{k}_1(H) = 0,\, \dots ,\, \tau^{k}_{j-1}(H) = 0\}.$$ \hfill $\diamondsuit$
\end{lemma}

\noindent {\bf Example 3.1.}
Let $M$ be a domain in $\R^2$, bounded by a smooth simple curve $\d M$, given by an equation $\{z(x_1, x_2) = 0\}$. We assume that $0$ is a regular value of the smooth function $z: \R^2 \to \R$ and that $z \leq 0$ defines $M$.  

Since $\R^2$ carries the Eucledian metric $g = g_{\mathsf E}$, the exponential map is given by the formula $\exp_{\vec x}(\vec w) = \vec x + \vec w$, where $\vec x \in \R^2$ and $\vec w \in T_x(\R^2) \approx \R^2$.

For each point $\vec x= (x_1, x_2)$, we consider the Taylor expansion $T(\vec x, \vec w)$ of the function $H(\vec x, \vec w) = z(\vec x+\vec w)$ as a sum:
$$T(\vec x, \vec w) = \tau_0(\vec x) + \tau_1(\vec x, \vec w) + \tau_2(\vec x, \vec w) + \tau_3(\vec x, \vec w) + \dots ,$$
where $w = (w_1, w_2) \in T_x(\R^2)$, and the summands are of the form:
\begin{eqnarray}
\tau_0(\vec x) &=& z(\vec x), \nonumber \\
\tau_1(\vec x, \vec w) & = & \tau^1_{1}(\vec x) w_1 + \tau^1_{2}(\vec x) w_2, \nonumber \\
\tau_2(\vec x, \vec w) & = & \tau^2_{11}(\vec x) w_1^2 + \tau^2_{12}(\vec x) w_1w_2+ \tau^2_{22}(\vec x) w_2^2, \nonumber \\
\tau_3(\vec x, \vec w) & = &  \tau^3_{111}(\vec x) w_1^3 + \tau^3_{112}(\vec x) w_1^2w_2 + \tau^3_{122}(\vec x) w_1w_2^2 + \tau^3_{222}(\vec x) w_2^3. \nonumber 
\end{eqnarray}


The manifold $SM$ is diffeomorphic to the product $M \times S^1$. The boundary $\d(SM)$  is a $2$-torus, given by the equation $\{\tau_0(\vec x) = 0\}$.

The locus $\d_2SM(v^g) = S(\d M) \subset \d(SM)$ is given by the homogeneous in $w_1, w_2$ equations
\begin{eqnarray}
\tau_0(\vec x) &= & 0  \nonumber \\
\tau^1_{1}(\vec x) w_1 + \tau^1_{2}(\vec x) w_2 & = & 0, \nonumber
\end{eqnarray}
where $(w_1, w_2)$ is a unitary vector.

In $\d_1^+(SM)(v^g)$, we get $\tau^1_{1}(\vec x) w_1 + \tau^1_{2}(\vec x) w_2 \geq 0.$

For a \emph{boundary generic} relative to $g_{\mathsf E}$ curve $\d M$, the locus $\d_2SM(v^g)$ is represented by two smooth parallel loops. They separate $\d(SM)$ into two annuli, $\d_1^+(SM)(v^g)$ and $\d_1^-(SM)(v^g)$.

The locus $\d_3(SM)(v^g) \subset \d_2(SM)(v^g)$ is given by the homogeneous equations 
\begin{eqnarray}
\tau_0(\vec x) & = & 0,  \nonumber \\
\tau^1_{1}(\vec x) w_1 + \tau^1_{2}(\vec x) w_2 & = & 0,  \nonumber \\
\tau^2_{11}(\vec x) w_1^2 + \tau^2_{12}(\vec x) w_1w_2+ \tau^2_{22}(\vec x) w_2^2 & = & 0,\nonumber 
\end{eqnarray}
\noindent where $(w_1, w_2)$ is a unitary vector. 

For a boundary generic relative to $g_{\mathsf E}$ curve $\d M$, the locus $\d_3(SM)(v^g)$ is a collection of an even number of points.

The locus $\d_4(SM)(v^g)$ is given by adding the equation $\tau^3(\vec x, \vec w) = 0$  
to the previous homogeneous system of three equations. The resulting system must have only the trivial solution $w_1 = 0, w_2 = 0, w_3 = 0$ for all $\vec x \in \d M$.
\smallskip

For a boundary generic relative to $g_{\mathsf E}$ curve $\d M$, the differential $1$-form $d\tau_0 \in T^\ast(SM)$ does not vanish at the points of  the surface $\{\tau_0 = 0\}$, the differential $2$-form $d\tau_0 \wedge d\tau_1 \in \bigwedge^2 T^\ast(SM)$ does not vanish at the points of the curve $\{\tau_0 = 0, \tau_1 = 0\}$, and the differential $3$-form $d\tau_0 \wedge d\tau_1 \wedge d\tau_2 \in \bigwedge^3 T^\ast(SM)$ does not vanish at the points of the finite locus $\{\tau_0 = 0, \tau_1 = 0, \tau_2 = 0\}$.
  
So a boundary generic relative to $g_{\mathsf E}$ curve $\d M$ may be either strictly convex, or the union of several arcs, along which the strict concavity and convexity of $\d M$ alternate. These arcs are bounded by finitely many points of cubic inflection. No tangent to $\d M$ line has tangency of an order that exceed $3$.
\hfill $\diamondsuit$

\smallskip

The validity of the next conjecture would imply the half of Conjecture \ref{conj9.1}, concerned with the space of boundary generic metrics  being dense in the space of all metrics. Perhaps, Lemma \ref{lem9.5a} could be useful in validating the conjecture.

\begin{conjecture}\label{conj9.2} Let $M$ be an open smooth $n$-manifold with a Riemannian metric $g$.  Consider a smooth compact submanifold $\a: N \subset M$ of codimension one. Then there is a small smooth isotopy $A: N \times [0, 1] \to M$ of the imbedding $\a$, such that the metric $g$ is boundary generic\footnote{That is, the geodesic flow $v^{g}$ on $SM$ is boundary generic with respect to the submanifold $\Sigma N$ that is produced by restricting the spherical bundle $SM \to M$ to $A(N \times \{1\})$.} with respect to the submanifold $A(N \times \{1\})$.  \hfill $\diamondsuit$
\end{conjecture} 

The next proposition claims that these refined stratified convexity/concavity properties of the boundary $\d M$ with respect to the metric $g$ on $M$ can be recovered from the scattering data.

\begin{corollary}\label{cor9.6} Assume that a smooth compact and connected $n$-manifold $M$ admits a geodesically boundary generic Riemannian metric $g$ of the gradient type. Then the scattering map $$C_{v^g}: \d_1^+(SM) \to \d_1^-(SM)$$  allows for a reconstruction of the loci $\{\d_j^\pm SM(v^g)\}_{j \in [1, 2n-1]}$. 

As a result, for each $j \in [2, 2n-1]$, the locus $\{(x, w) \in SM|_{\d M}\}$, such that the germ of the geodesic curve $\g \subset M$ through $x$ in the direction of $w$ is tangent to $\d M$ with the multiplicity $j -1$,  can be reconstructed from  the scattering map. 
\end{corollary} 

\begin{proof} By the proof of Holography Theorem 3.1 from \cite{K4}, the causality map $C_v: \d_1^+X(v) \to \d_1^-X(v)$ of any  boundary generic traversing field $v$ on a smooth compact manifold $X$ with boundary allows for a reconstruction of all the strata $\{\d_j^\pm X(v)\}$. With this fact in hand, the claims of the corollary are the immediate implications  of Lemma \ref{lem9.4} and  Theorem \ref{th9.2}. 
\end{proof}

Assume that a metric $g$ on $M$ is geodesically boundary generic (see Definition \ref{def9.3}). For an oriented geodesic curve $\g \subset M$, consider its intersections $\{x\}$ with $\d M$. We denote by $m(x)$ the multiplicity of an intersection point $x \in \g \cap \d M$ (see the proof of Lemma \ref{lem9.5}). We define the \emph{multiplicity} of $\g$ by the formula $$m(\g) =_{\mathsf{def}} \sum_{x \in \g \cap \d M} m(x),$$ and the \emph{reduced multiplicity} of $\g$ by the formula $$m'(\g) =_{\mathsf{def}} \sum_{x \in \g \cap \d M} (m(x) - 1)$$ (see \cite{K2} for the properties of these quantities).\smallskip

The following corollary describes the ways in which  geodesic arcs can be inscribed in $M$, provided that the metric on $M$ is \emph{traversally  generic}, that is, the geodesic field $v^g$ is traversally generic on $SM$ (see Definition 3.2 from \cite{K2} for the notion of traversally generic vector field). For example, if a metric on a surface $M$ is traversally generic, then any geodesic curve may have two simple tangencies to $\d M$ at most. Similarly, any geodesic curve may have one cubic tangency to $\d M$ at most. No tangencies of multiplicity $\geq 4$ are permitted.

\begin{corollary}\label{cor9.7} Assume that a $n$-dimensional compact smooth manifold $M$ admits a traversally  generic Riemannian metric $g$. Then any geodesic curve in $M$ has $2n -2$ simple points of tangency to the boundary $\d M$ at most. 

In general, any geodesic curve $\g \subset M$ interacts with the boundary $\d M$ so that the multiplicity and the reduced multiplicity of $\g$ satisfy the inequalities:  $$m(\g) \leq 4n - 2\;\; \text{and} \;\; m'(\g) \leq 2n - 2.$$

Moreover, these inequalities hold for any metric $g'$ sufficiently close, in the $C^\infty$-topology, to $g$.  
\end{corollary} 

\begin{proof} Since, by Lemma \ref{lem9.5}, the multiplicity of tangency of a geodesic curve $\g$ to $\d M$ at a point $x$ and the multiplicity of tangency of its lift $\tilde\g \subset SM$ to $\d(SM)$ at the point $(x, \dot\g(x))$ are equal, the claim follows from the second bullet in Theorem 3.5 from \cite{K2}. 

The last claim follows from two facts: (1) the field $v^g$ depends smoothly on $g$, and (2) the space $\mathcal V^\ddagger(SM)$ is open in the space $\mathcal V(SM)$.
\end{proof}

In one special case of $(M, g)$,  the traversal genericity of the geodesic flow on $SM$ comes ``for free" at the expense of a \emph{very restricted topology} of $M$.  So a random manifold $M$ does not admit a non-trapping metric $g$ such that $\d M$ is convex!

\begin{corollary}\label{cor9.8} Let $M$ be a compact connected Riemannian $n$-manifold with boundary. Assume that the boundary $\d M$ is \emph{strictly convex} with respect to a metric $g$ of the gradient type on $M$. \smallskip

Then geodesic field $v^g$ on $SM$ is traversally  generic. \smallskip

Moreover, the manifold $SM$ is diffeomorphc to the product $D(\d M) \times [0,1]$, the corners in the product being rounded. Here $D(\d M)$ denotes the tangent  $(n-1)$-disk bundle of $\d M$.
\end{corollary} 

\begin{proof} For such metric $g$,  any geodesic curve $\g \subset M$, tangent to $\d M$, is a singleton. Therefore $\d_2SM(v^g) =  \d_2^-SM(v^g)$ --- the field $v^g$ on $SM$ is convex. Under these assumptions, the geodesic field $v^g$ on $SM$ is traversally  generic, since no strata $\{\d_jSM(v^g)\}_j$ interact, with the help of the geodesic flow, trough the bulk $SM$. Also, for a strictly convex $\d M$ and $g$ of the gradient type,  
$$\d_1^+SM(v^g) = \{(x, v) |\;  x \in \d M, \, v\; \text{points inside}\; M\}.$$ Thus $\d_1^+SM(v^g)$ fibers over $\d M$ with a fiber being the hemisphere $D^{n-1}_+ \subset S^{n-1}$ of dimension $n - 1 = \dim(\d M)$. This fibration is isomorphic to the unit disk tangent bundle of $\d M$. Hence, by Lemma 4.2 from \cite{K1},  the manifold $SM$ must be diffeomorphc to the product $D(\d M) \times [0,1]$, the corners in the products being rounded.  As a result, $SM$ fibers over $\d M$ with the fiber $D^n$ and over $M$ with the fiber $S^{n-1}$. This property of $SM$ puts severe restrictions on the topology  of $M$: in particular, the space $SM$ must be homotopy equivalent to $\d M$. 

Note that $M = D^n$ has the desired property: $M \times S^{n-1} \approx \d M \times D^n$.
\end{proof}

The next corollary should be compared with Theorem D from \cite{Cr2}, in a way, a stronger than Corollary \ref{cor9.9} result, but using more data (the distance ``across $M$" between points of $\d M$).   
 
\begin{corollary}\label{cor9.9} Let $M \subset \H^n$ be a codimension $0$ compact smooth submanifold of the hyperbolic space, such that the metric $g = (g_{\H})|_M$ is geodesically boundary generic.

Then the metric $g$ is of the gradient type, and the scattering map $C_{v^{g}}: \d_1^+(SM) \to \d_1^-(SM)$  allows for a reconstruction of the pair $(SM, \mathcal F(v^{g}))$, up to a homeomorphism (a diffeomorphism when the property $\mathsf A$ is valid) of $SM$ which is the identity on $\d(SM)$. 
\end{corollary}

\begin{proof} In view of Corollary \ref{cor9.4} and Example 2.3, 
the hyperbolic metric on $M$ is of the gradient type. By Theorem \ref{th9.2}, the corollary follows.
 \end{proof}
\smallskip

One might hope that, for a ``random" (bumpy) metric $g \in \mathcal G^\dagger(M)$ on $M$, the reconstruction $C_{v^g} \Rightarrow (M, g)$  is possible, up to the action of diffeomorphisms that are the identity on $\d M$ (see \cite{Cr}, \cite{Cr1}, \cite{We}, and especially \cite{SU4} and \cite{SUV3} for special interesting cases of such reconstructions). At the moment, this is just a wishful thinking, weakly supported by \cite{Mat}. Here is the main difficulty in reaching this conclusion, as we see it now from the holographic viewpoint. Although, by Theorem \ref{th9.2}, $C_{v^g}$ allows for a reconstruction of the pair $(SM, \mathcal F(v^g))$, the fibration map $SM \to M$ seems to resist a reconstruction. In other words, the scattering map $C_{v^g}$ ``does not know how to project the geodesic flow trajectories in $SM$ to the geodesic curves in $M$".  Perhaps, some additional structure should be brought into the play.
 \smallskip

So, in the absence of a faithful reconstruction of the geometry $g$ on $M$ from the scattering data $C_{v^g}$,  with a mindset of a ``humble topologist", we will settle for less:
 
\begin{theorem}\label{th9.4} Assume that a smooth compact connected $n$-manifold $M$ with boundary admits a boundary generic Riemannian metric $g$ of the gradient type.  Then the following statements are valid:
\begin{itemize}
\item The geodesic scattering map $C_{v^g}: \d_1^+(SM) \to \d_1^-(SM)$ allows for a reconstruction of the cohomology rings $H^\ast(M; \Z)$ and $H^\ast(M, \d M; \Z)$, as well as for the reconstruction of the homotopy groups $\{\pi_i(M)\}_{i < n}$. \smallskip

\item Moreover, the Gromov simplicial semi-norms $\|\sim\|_{\mathbf \D}$ on the vector spaces $H^\ast(M; \R)$ and on $H^\ast(M, \d M; \R)$ can be reconstructed form $C_{v^g}$. In particular,  the simplicial volume $\|[M, \d M] \|_{\mathbf \D}$ of the fundamental  cycle $[M, \d M]$ can be recovered form $C_{v^g}$. \smallskip

\item If, in addition, $M$ has a trivial tangent bundle, then the stable topological (smooth when the property $\mathsf A$ from Definition \ref{def9.4a} holds) type of $M$\footnote{Here we say that $M$ and $M'$ share the same stable topological (smooth) type, if $M \times S^{n-1}$ and $M' \times S^{n-1}$ are homeomorphic (diffeomorphic).} is also reconstructable from the scattering map. 
\end{itemize}
\end{theorem}

\begin{proof} Since, by Theorem \ref{th9.2}, the topological type of $SM$ can be reconstructed from the scattering map $C_{v^g}$, so is the cohomology/homology of $SM$ with arbitrary coefficients. Moreover, the Gromov simplicial semi-norms $\|\sim\|_{\mathbf \D}$ on $H^\ast(SM; \R)$ and on $H_\ast(SM; \R)$, being invariants of the homotopy type of the pair $(SM, \d(SM))$, can be recovered form $C_{v^g}$. 

Since $H^n(M; \Z) = 0$ due to the property $\d M \neq \emptyset$, the fibration $\pi: SM \to M$ admits a section $\s$. Thus $M$ is a retract of $SM$. 

The cohomology/homology spectral sequence of the spherical fiibration $\pi: SM \to M$ is trivial, again since $H^n(M; \Z) = 0$. As a result,  $$H^\ast(SM; \Z) \approx H^\ast(M; \Z) \oplus H^{\ast - n +1}(M; \Z).$$ Note that $H^\ast(M; \Z) \approx H^\ast(SM; \Z)$ for all $\ast < n-1$. Therefore $H^\ast(M; \Z)$, a direct summand of $H^\ast(SM; \Z)$, can be recovered from $H^\ast(SM; \Z)$ and thus from the scattering map $C_{v^g}$. 

Similar considerations are valid for $H^\ast(M, \d M; \Z)$, a direct summand of the homology $H^\ast(SM, \d(SM); \Z)$.

The simplicial semi-norm of a homology class does not increase under continuous maps of spaces \cite{G}. Therefore, with the help of $\s^\ast$, we get that the natural homomorphisms $$\pi^\ast: H^\ast(M; \R) \to H^\ast(SM; \R), \;\; \pi^\ast: H^\ast(M, \d M; \R) \to H^\ast(SM, \d(SM); \R),$$ induced by the map $\pi: SM \to M$, are \emph{isometries} with respect to the semi-norm  $\|\sim\|_{\mathbf \D}$. Hence, the Gromov simplicial semi-norms on $H^\ast(M, \d M; \Z)$ and $H^\ast(M; \Z)$ can be recovered form $C_{v^g}$ as well. In particular, the simplicial volume $\|[M, \d M] \|_{\mathbf \D}$ of the fundamental class $[M, \d M]$ can be recovered form the scattering data $C_{v^g}$.
\smallskip

The long exact homotopy sequence  of the fibration $SM \to M$ identifies $\pi_i(M)$ with $\pi_i(SM)$ for all $i < n$. As a result, the homotopy groups $\{\pi_i(M)\}_{i < n}$ can be recovered from the scattering map $C_{v^g}$ as well.

For a trivial tangent bundle $TM$, the ``reconstructible" space $SM$ is diffeomorphic to the product $M \times S^{n-1}$. Thus the stable smooth type of $M$ (the stabilization being understood as the multiplication with a sphere) is ``reconstructible" as well. In particular, the homotopy type of $M$ can be reconstructed from the scattering data.
\end{proof}

By Theorem \ref{th9.2}, if two scattering maps are smoothly conjugated, then the two metrics are geodesic flow topologically conjugated, provided they are boundary generic and of the gradient type. In general, we do not know when two metrics, which are geodesic flow topologically/strongly conjugate in the sense of Definition \ref{def9.4a}, are isometric/proportional (see \cite{Cr}, \cite{Cr1}, \cite{We} for the positive answers in special cases; for example, Croke proved that, for $n \geq 2$, the flat product metric on $D^n \times S^1$ is scattering rigid). However, for one family of special cases, dealing with metrics of negative sectional curvature, we get a pleasing answer in Theorem \ref{th9.6}. This theorem should be compared with more general results of a similar nature in \cite{SUV1}, \cite{SUV2}, \cite{SUV3}, obtained by powerful analytic techniques. \smallskip

In order to illustrate a connection between the geodesic flow topological conjugacy and isometry of Riemannian manifolds, let us  recall the famous Mostow Rigidity Theorem \cite{Most}.  Let $X$ and $Y$ be complete finite-volume hyperbolic $n$-manifolds with $n \geq 3$. If there exists an isomorphism  $\phi_\ast: \pi_1(X) \to \pi_1(Y)$ of the fundamental groups, then the Mostow Theorem claims that $\phi_\ast$ is induced by a unique isometry $\phi: X \to Y$. 

Here is  another formulation of the Mostow Rigidity Theorem, better suited for our goals:

\begin{theorem}\label{th9.5} {\bf (Mostow)} Let $(N_1, g_1)$ and $(N_2, g_2)$ be two closed locally symmetric Riemannian manifolds  with negative sectional curvatures and of dimension $\geq 3$.  

Given an isomorphism $\phi: \pi_1(N_1) \to \pi_1(N_2)$ of the fundamental groups, there exists a unique isometry $\Phi: (N_1, c\cdot g_1) \to (N_2, g_2)$, where $c > 0$ is a constant, so that the map $\Phi$ induces the isomorphism $\phi$ of the fundamental groups. \hfill $\diamondsuit$
\end{theorem}

In the  spirit of Theorem 1.3 from \cite{BCG} and \cite{CEK}, by combining Mostow Theorem \ref{th9.5} with Theorem \ref{th9.2}, we get the following result. It and its Corollary \ref{cor9.10} should be compared with Theorem D from \cite{Cr2}, which deals with domains in the hyperbolic space $\H^n$.

\begin{theorem}\label{th9.6} Let $n \geq 3$.  Consider two closed locally symmetric Riemannian $n$-manifolds, $(N_1, g_1)$ and $(N_2, g_2)$, with negative sectional curvatures. Let a connected manifold $M_i$ ($i = 1, 2$) be obtained from $N_i$ by removing the interior of a smooth codimension zero submanifold $U_i \subset N_i$, such that the induced homomorphism $\pi_1(M_i) \to \pi_1(N_i)$ of the fundamental groups is an isomorphism\footnote{For example, by a general position argument, this the case when $U_i$ has a spine of codimension $3$ at least. In particular, $U_i$ may be a disjoint union of $n$-balls.}.

Assume that the restriction of the metric $g_i$ to $M_i$ is boundary generic and of the gradient type\footnote{Thanks to Theorem \ref{th9.0}, the latter hypotheses is not restrictive; by Conjecture \ref{conj9.2}, the boundary generic hypotheses does not seem restrictive either.}. Assume also that the two geodesic scattering maps $$C_{v^{g_1}}: \d_1^+(SM_1) \to \d_1^-(SM_1), \quad C_{v^{g_2}}: \d_1^+(SM_2) \to \d_1^-(SM_2)$$ are conjugated via a smooth diffeomorphism $\Phi^\d: \d (SM_1) \to \d (SM_2)$\footnote{Thus the boundaries $\d U_1$ and $\d U_2$ are stably diffeomorphic.}.\smallskip

Then $\Phi^\d$ determines a unique diffeomorphism $\phi: N_1 \to N_2$ such that $\phi^\ast(g_2) = c \cdot g_1$ for a constant $c > 0$.
\smallskip
\end{theorem}

\begin{proof} By Theorem \ref{th9.2} (see also Theorem \ref{th9.4}), the spaces $SM_1$ and $SM_2$ are homeomorphic via a homeomorphism $\Phi$ which is an extension of $\Phi^\d$. In particular, $\Phi$ induces an  isomorphism $\Phi_\ast: \pi_1(SM_1) \to \pi_1(SM_2)$ of the fundamental groups. Each space $SM_i$ fibers over $M_i$ with the spherical fiber $S^{n-1}$. Thus, when $n \geq 3$, the fundamental groups $\pi_1(M_1)$ and  $\pi_1(M_2)$ are isomorphic via an isomorphism $\phi_\ast: \pi_1(M_1) \to \pi_1(M_2)$, which is determined by $\Phi_\ast$. In fact, $\phi_\ast$ is induced by taking a section $\s_1: M_1 \to SM_1$ (the bundle $TM_1 \to M_1$ admits a non-vanishing section since $\d M_1 \neq \emptyset$), composing $\s_1$ with $\Phi$, and then applying the projection $SM_2 \to M_2$. 

By the hypotheses, the inclusion homomorphism $\pi_1(M_i) \to \pi_1(N_i)$ is an isomorphism. Therefore, $\pi_1(N_1)$ and $\pi_1(N_2)$ are isomorphic via an isomorphism $\hat\phi_\ast$, which is determined by $\phi_\ast$ (and thus eventually by $\Phi$).   By Mostow's Rigidity Theorem  \ref{th9.5}, for a given $\hat\phi_\ast$, there exists a unique diffeomorphism $\hat\phi: N_1 \to N_2$ such that $\hat\phi^\ast(g_2) = c \cdot g_1$, where $c$ is a positive constant.

We claim that the diffeomorphism $\hat\phi: N_1 \to N_2$ is actually determined by the diffeomorphism $\Phi^\d: \d(SM_1) \to \d(SM_2)$. Indeed, consider the obvious map $\Gamma_i: SM_i \to \mathcal T(v^{g_i})$, where $\mathcal T(v^{g_i})$ denotes the trajectory space of the geodesic $v^{g_i}$-flow. Since $v^{g_i}$ is traversing, $\Gamma_i$ is a quasifibration with contractible fibers. Therefore, $(\Gamma_i)_\ast: \pi_1(SM_i) \to \mathcal \pi_1(\mathcal T(v^{g_i}))$ is an isomorphism. By Theorem \ref{th9.2}, we get a commutative diagram $\Phi^\mathcal T \circ \Gamma_1 = \Gamma_2 \circ \Phi$, where $\Phi^\mathcal T: \mathcal T(v^{g_1}) \to \mathcal T(v^{g_2})$ is a homeomorphism, which is determined by $\Phi^\d$. In fact, $\Phi^\mathcal T \circ \Gamma_1^\d = \Gamma_2^\d \circ \Phi^\d$, where $\Gamma_i^\d = \Gamma_i |_{\d(SM_i)}$. Therefore $\Phi_\ast: \pi_1(SM_1) \to \pi_1(SM_2)$ is determined by $\Phi^\d$. As a result, $\phi_\ast: \pi_1(M_1) \to \pi_1(M_2)$ and thus $\hat\phi_\ast: \pi_1(N_1) \to \pi_1(N_2)$ are determined by $\Phi^\d$. So by Mostow's Rigidity Theorem  \ref{th9.5}, the corresponding unique diffeomorphism $\hat\phi: N_1 \to N_2$ that delivers the isometry between $g_2$ and $c\cdot g_1$ is determined by  $\Phi^\d$.
\end{proof}

\noindent {\bf Remark 3.4.} Under the hypotheses of Theorem \ref{th9.6}, consider the case when the diffeomorphism $\Phi^\d$ that conjugates the scattering maps is the identity. Let $U_1 = U_2$ be a $n$-ball, and $M_1 = M_2 =_{\mathsf{def}} M$. Then $\d(SM) \approx S^{n-1} \times S^{n-1}$ and  $C_{v^g}$ is a map whose source and target are both canonically diffeomorphic to $S^{n-1} \times D^{n-1}$ (with the help of the involution that orthogonally reflects tangent vectors with respect to the boundary $\d M$). By Theorem \ref{th9.6}, the scattering map $C_{v^g}: S^{n-1} \times D^{n-1} \to S^{n-1} \times D^{n-1}$ allows to reconstruct the pair $(N, g)$, up to a positive scalar.  In other words, the scattering maps distinguishes (up to a constant conformal factor) between different locally symmetric spaces of negative sectional curvature.
\hfill $\diamondsuit$
\smallskip

The next immediate corollary of Theorem \ref{th9.6} is inspired by the image of geodesic motion of a bouncing particle in the complement $M$ to a number of disjoint balls, placed in a closed hyperbolic manifold $N$ of a dimension greater than two. The balls a placed so ``dense" in $N$ that every geodesic curve hits some ball. Under these assumptions, the collisions of a probe particle in $M$ with the boundary $\d M$ ``feel the shape of $N$". 

\begin{corollary}\label{cor9.10} Let $(N_1, g_1)$ and $(N_2, g_2)$ be two closed hyperbolic manifolds of dimension $n \geq 3$. Let $M_i$ be produced from $N_i$ by removing the interiors of $k$ disjointed smooth $n$-balls, so that the geodesic flow on each $M_i$ is boundary generic\footnote{This is the case when the boundaries of the balls are strictly  convex in $N_i$.} and of the gradient type. 

If the scattering maps  $C_{v^{g_1}}$ and $C_{v^{g_2}}$ are conjugated with the help of a smooth diffeomorphism $\Phi^\d: \d(SM_1) \to \d(SM_2)$, then $(N_1, g_1)$ is  isometric to $(N_2, g_2)$. \hfill $\diamondsuit$
\end{corollary}

Let $(N, g)$ be a closed locally symmetric Riemannian manifold with negative sectional curvature, and let $U \subset N$ be a smooth codimension zero submanifold. Take $M = N \setminus \textup{int}(U)$ and assume that $g|_M$ is boundary generic and of the gradient type. 

Recall that, by the Mostow Theorem, the isometry group $\mathsf{Iso}(N, g)$ is isomorphic to the group $\mathsf{Aut}(\pi_1(N))$ of automorphisms of the fundamental group $\pi_1(N)$. 

For an isometry $\psi: N \to N$, put $M_\psi  =_{\mathsf{def}} N \setminus \textup{int}(\psi(U))$. Evidently any isometry $\psi: N \to N$ induces a diffeomorphism $\Psi: SN \to SN$ which commutes with the geodesic flow on $SN$. Moreover, $\Phi =_{\mathsf{def}} \Psi|: SM \to SM_\psi$ strongly conjugates the geodesic flows, delivered by the metrics $g^M  = g|_{M}$ and $g^{M_\psi}  =g|_{M_\psi}$. As a result, the scattering maps $C_{v^{g^M}}$ and $C_{v^{g^{M_\psi}}}$ are conjugated with the help of the diffeomorphism $\Phi^\d = \Phi | : \d(SM) \to \d(SM_\psi)$. 

Is this construction the only way in which the scattering maps on the complements of isometric domains in $N$ can be conjugated?   More accurately, we ask the following question.

\begin{question} \emph{Let $(N, g)$ be a closed locally symmetric Riemannian manifold with negative sectional curvature. Consider two smooth codimension zero submanifolds $U, V \subset N$.  Let $\psi: U \to V$ be an isometry and denote by $\Psi: SU \to SV$ the diffeomorphism, induced by $\psi$. Form the submanifolds  $M = N \setminus \textup{int}(U)$ and $L = N \setminus \textup{int}(V)$.}

\emph{Assume that the metrics $g_M =_{\mathsf{def}} g|_M$ and $g_L =_{\mathsf{def}} g|_L$ are boundary generic and of the gradient type. Also assume that the scattering maps $C_{v^{g_M}}$ and $C_{v^{g_L}}$ are conjugated with the help of a diffeomorphism $\Psi^\d =_{\mathsf{def}} \Psi|_{\d(SU)}$.}

Can we conclude that there exists an isometry $\tilde\psi: N \to N$ such that $\tilde\psi|_U = \psi$? \hfill $\diamondsuit$
\end{question}

\section{The Inverse Geodesic Scattering Problem in Presence of Length Data}

Now we will enhance the scattering data $C_{v^g}$ by adding information about the length (equivalently, travel time) along each $v^{g}$-trajectory. This new combination of data is commonly called ``\emph{lens data}".

\begin{definition}\label{def9.?} Let $g$ be a smooth metric on a compact manifold $M$, and $F: SM \to \R$ a smooth function such that $dF(v^g) > 0$. For each point $w \in \d_1^+(SM)$, let $\tilde\g_w$ be the segment of the  $v^g$-trajectory that connects $w$ and $C_{v^g}(w)$. We denote by $\g_w$ its image under the projection $\pi: SM \to M$, and by $l_g(w)$ the length of the geodesic arc $\g_w$.\smallskip

We say that the function $F$ is \emph{balanced}  if, for each point $w \in \d_1^+(SM)$, the variation $F(C_{v^g}(w)) - F(w)$ is equal (up to an universal constant) to the arc  length $l_g(w)$. \smallskip

We call a gradient type Riemannian metric $g$ \emph{balanced}, if $v^g$ admits a balanced Lyapunov function $F$. \hfill $\diamondsuit$
\end{definition}

\noindent{\bf Example 4.1.} Let $M$ be a compact smooth domain in the hyperbolic space $(\H^n, g_\H)$. Using a modification $F$ of the Lyapunov  function $d_\H$ from Example 2.3, we will see that the restriction of the hyperbolic metric to $M$ is $F$-balanced. Let us sketch the argument, based on the Poincar\`{e} model of $\H^n$. Any geodesic in $\H^n$ is orthogonal to the virtual boundary $S_\infty^{n-1}$ of $\H^n$.  Therefore, using the compactness of $M$, there exists a big Euclidean ball $B \supset M$, whose center is in $M$ and such that any geodesic curve $\g$ in $\H^n$, which intersects $M$, has two distinct transversal intersections with the boundary $\d B$. The orientation of $\g$ picks a single point $a_\g \in \d B \cap \g$, where the velocity vector points inside of $B$. For any $w = (m, v) \in SM$,  we consider the geodesic $\g_w$ through $m$ in the direction of $v$ and a point $a_{\g_w} \in \d B$. We introduce the smooth function $F: SM \to \R$ by the formula $F(w) =_{\mathsf{def}} \textup{dist}_\H(a_{\g_w}, m)$. Evidently, $dF(v^{g_\H}) > 0$ and the variation of $F$ along any geodesic arc is the length of that arc. 

By a very similar argument, any compact domain $M$ in the Eucledian space $\mathsf E^n$ inherits a flat balanced metric of the gradient type.
\smallskip

These observations can be generalized. Let $M$ be a compact smooth codimension zero submanifold of a compact connected Riemannian manifold $(N, g)$ with boundary. Assume that every geodesic $\g$ in $N$, such that $\g \cap M \neq \emptyset$,  intersects $\d N$ transversally at a pair of points. Then, by a construction, similar  to the previous one, the metric $g|_M$ is of the gradient type and balanced.
\hfill $\diamondsuit$ 
\smallskip

For a traversing vector field $v$ on compact manifold $X$, let $I_x^y$ denote the segment of the $v$-trajectory $\g$ that is bounded by a pair of points $x, y \in \g$. Given a $1$-form $\theta$ on $X$, we use the notation ``$\int_x^y \theta$" for the integral $\int_{I_x^y}\theta$.

\begin{lemma}\label{lem9.x} Let $(X, \mathsf g)$ be a compact smooth Riemannian manifold with boundary. Let $v \neq 0$ be a smooth traversing and boundary generic vector field, and let $\a$ be a smooth $1$-form on $X$ such that $\a(v) > 0$. 
 
 Assume that for each $x \in \d_1^+X(v)$, the integrals $\int_x^{C_v(x)}  \|v\|_{\mathsf g}\, d\mathsf g$ \footnote{Here $C_v: \d_1^+X(v) \to \d_1^-X(v)$ denotes the $v$-generated causality map, and $d\mathsf g$ denotes the measure on the trajectory $\g_x$ induced by the metric $\mathsf g|_{\g_x}$.}  and $\int_x^{C_v(x)}\a$ are equal.

Then there exists a homeomorphism  $\Theta: X \to X$ with the following properties: 
\begin{itemize}
\item $\Theta|_{\d X}$ is the identity,
\item each $v$-trajectory $\g$ is invariant under $\Theta$, 
\item each restriction $\Theta |: \g \to \g$ is a smooth diffeomorphism, 
\item  $(\Theta|_\g)^\ast(\a)(v) = \|v\|_{\mathsf g}$. 
\begin{eqnarray}\label{eq9.?}
\end{eqnarray}
\end{itemize}

If $v$ is such that any $v$-trajectory $\g$ is either transversal to $\d X$ at \emph{some} point from $\g \cap \d X$, or $\g \cap \d X$ is a singleton of multiplicity $2$, then the homeomorphism $\Theta$ is a smooth diffeomorphism. 
\end{lemma}

\begin{proof} Consider two strictly monotone smooth $y$-functions on the arc $I_x^{C_v(x)}$: $$K(y) = \int_x^y \a, \qquad L(y) = \int_x^y \|v\|_g\, d\mathsf g.$$ 

Put $\Theta(y) =_{\mathsf{def}} (K^{-1} \circ L)(y)$. In other words, for each $x \in \d_1^+X(v)$ and $y \in I_x^{C_v(x)}$, $\Theta(y)$ is well-defined by the identity $$\int_x^y \a = \int_x^{\Theta(y)} \|v\|_\mathsf g\, d\mathsf g.$$ By the lemma hypotheses, $\Theta(x) = x$ and $\Theta(C_v(x)) = C_v(x)$ for all $x \in \d_1^+X(v)$. Thus, $\Theta(I_x^{C_v(x)}) = I_x^{C_v(x)}$. As a result, $\Theta$ is the identity on the boundary $\d X$. 

For any  $y' \in I_{x}^{C_v(C_v(x))}$, put 
$$\tilde K(y') =_{\mathsf{def}} \int_x^{y'} \a = L + \int^{y'}_{C_v(x)}\a, \qquad \tilde L(y') =_{\mathsf{def}}  \int_x^{y'}\|v\|_{\mathsf g}\, d\mathsf g = L + \int^{y'}_{C_v(x)}\|v\|_{\mathsf g}\, d\mathsf g,$$ where $L$ denotes  $\int_x^{C_v(x)} \|v\|_{\mathsf g}\, d\mathsf g$. Since $\tilde K$ and $K$, as well as $\tilde L$ and $L'$, differ by the same parallel shift $L$, we have $(K^{-1} \circ L)(y') = (\tilde K^{-1} \circ \tilde L)(y')$.
By the continuity of $\mathsf g$ and $\a$, we get that the $y'$-function $\tilde K^{-1} \circ \tilde L$ is continuous in the vicinity of $C_v(x) \in \g_x$. Thus $\Theta(y)$ and $\Theta(y')$ are close for any two points $y, y' \in \g_x$ that are sufficiently close to the point $C_v(x)$. 

By a similar inductive argument in $i \in [1, k]$, applied to the intervals $\{[x, C_v^{(\circ i)}(x)]\}_{i \in [1, k]}$, where $C_v^{(\circ i)}$ stands for the $i$-th iteration of the partially defined map $C_v$, we get that $\Theta|: \g_x \to \g_x$ is a homeomorphism. 

Note that the  derivative $\Theta' > 0$ in $I_x^{C_v(x)}$ since, by the definition of $\Theta$,  $\Theta'(y) = \|v(y)\|_{\mathsf g} / \a(v(y))$ for all $y \in \g_x$.  
Therefore, the restriction of $\Theta$ to any $v$-trajectory $\g$ is a smooth diffeomorphism.\smallskip

In order to show that $\Theta: X \to X$ is a homeomorphism, we embed $X$ into an open manifold $\hat X$ and extend smoothly the field $v$, the form $\a$, and the metric $\mathsf g$ into a open neighborhood $U$ of $X$ in $\hat X$ so that the extensions $\hat v$, $\hat \a$,  $\hat{\mathsf g}$ have the property $\hat\a(\hat v) > 0$ in $U$. In what follows, we may adjust the size of $U$ according to the needs of the arguments. So we will treat $\hat X = U$, $\hat v$, $\hat \a$,  $\hat{\mathsf g}$ as a germs. 

Note that any $v$-trajectory $\g$ is contained in a unique $\hat v$-trajectory $\hat \g$.\smallskip

In the next paragraph, we will rely on few basic facts about boundary generic traversing flows (see \cite{K3}). Let $\g_0$ be a $v$-trajectory. Since $v$ is boundary generic, the intersection $\g_0 \cap \d X$ is a finite ordered collection of points $\{x_1, x_2, \dots, x_k\}$.  So $C_v(x_i) = x_{i+1}$ for all $i \in [1, k-1]$, unless $\g_0 \cap \d X$ is a singleton ($k = 1$), in which case $C_v(x_1) = x_1$.

Let $V$ be a cylindrical neighborhood of $\hat\g_0$ in $\hat X$ which consists of $\hat v$-trajectories. By choosing an appropriate $U$, we may assume that $\hat\g_0 \cap X = \g_0$. We may also assume that $V$ admits a system of coordinates $(u, \vec w)$ such that the $\hat v$-trajectories are given by the equations $\{\vec w = \vec {const}\}$. For any $x \in \hat\g$, we denote by $S_x$ the hypersurface $\{u = u(x)\}$.

For any $\hat\g \subset V$, the intersection $\hat\g \cap \d X$ consists of finitely many points, while $\hat\g \cap X$ consists of finitely many closed intervals or singletons. Picking $V$ sufficiently narrow, the set $\hat\g \cap \d X$ may be divided into at most $k$ disjoint (possibly empty) subsets $A_{\hat\g, i}$ that correspond to the elements of $x_i \in \g_0 \cap \d X$.  The cardinality of each set $A_{\hat\g, i}$ does not exceed $m(x_i)$, the multiplicity of tangency of $\g_0$ to $\d X$ at $x_i$, and  $\#(A_{\hat\g, i}) \equiv m(x_i) \mod 2$. 

Let $\{S_i =_{\mathsf{def}} S_{x_i}\}_{1\leq i \leq k}$ be disjoint transversal sections of the $\hat v$-flow in the vicinity of $\g_0 \subset V$. By the choice of $V$ and by the construction of the non-empty sets $A_{\hat\g, i} \subset \hat\g$, the points of  $A_{\hat\g, i}$ are located in the vicinity of the point $y_i = S_i \cap \hat\g$. \smallskip

Let us compare $\Theta(x)$ and $\Theta(y)$ for a pair of points $x \in \g_0$ and $y \in \hat\g \cap X$ that are close in $X$. We may assume that,  for some $i$,  $x \in [x_i, x_{i+1}) \subset \hat\g_0$ and that $z =_{\mathsf{def}} S_x \cap \hat\g$ and $y$ are close in $\hat\g$. By the previous argument, $\Theta(z)$ and $\Theta(y)$ are close, so it suffices to compare $\Theta(z) = (K^{-1} \circ L)(z)$ and $\Theta(x) = (K^{-1} \circ L)(x)$.  Here, by the definition, $K(x) = \int_{x_i}^x \a, \; L(y) = \int_{x_i}^x \|v\|_\mathsf g\, d\mathsf g$, and $K(z) = \int_{z^\star}^z \a, \; L(z) = \int_{z^\star}^z \|v\|_\mathsf g\, d\mathsf g$, where $z^\star$ is the lowest point in the connected component of the set $\hat\g \cap X$ that contains $z$. By the choice of $V$,  for some $j \leq i$, $z^\star$ is a point of the set $A_{\hat\g, j}$, and thus is close to both $x_j \in \g_0$ and to the point $z_j^\star =_{\mathsf{def}} S_j \cap \hat\g$.

By the continuity of $\a$ and $\mathsf g$ and by the choice of $V \supset \g_0$, the value $K_\dagger(x) = \int_{x_j}^x \a$ is close to the value $K_\ddagger(z) = \int_{z^\star_j}^z \a = \int_{z^\star_j}^{z^\star} \a + K(z)$ and the value  $L_\dagger(x) = \int_{x_j}^x \|v\|_\mathsf g\, d\mathsf g$ is close to $L_\ddagger(z) = \int_{z^\star_j}^z \|v\|_\mathsf g\, d\mathsf g =  \pm \int_{z^\star_j}^{z^\star} \|v\|_\mathsf g\, d\mathsf g + L(z)$. Therefore, $(K_\dagger^{-1} \circ L_\dagger)(x)$ is close to $(K_\ddagger^{-1} \circ L_\ddagger)(z)$. Since, by the lemma hypotheses, $\int_{x_j}^{x_i} \a = \int_{x_j}^{x_i} \|v\|_\mathsf g\, d\mathsf g$, we get that  $(K_\dagger^{-1} \circ L_\dagger)(x)$ is close to $(K^{-1} \circ L)(x)$. Since $\int_{z^\star_j}^{z^\star} \a$ and $\int_{z^\star_j}^{z^\star} \|v\|_\mathsf g\, d\mathsf g$ are small, we conclude that $(K_\ddagger^{-1} \circ L_\ddagger)(z)$ is close to  $(K^{-1} \circ L)(z)$.  As a result, $\Theta(x)$ and $\Theta(z)$ are close in $V$. Thus $\Theta: X \to X$ is continuous. By the same token, $\Theta^{-1}: X \to X$ is continuous as well. \smallskip

The argument validating that $\Theta$ is a smooth diffeomorphism (under the Lemma \ref{lem9.x} hypotheses, the last paragraph), is similar the the one in the proof of Theorem 3.1 \cite{K4}. It employs that the boundary $\d X$, at a point of transversal intersection $\g \cap \d X$, is a smooth section $\s$ of the $v$-flow, together with the hypotheses that the metric $\mathsf g$ and the form $\a$ are smooth. This implies that the transformation $\Theta$, defined by the formula $\int_{\g_x \cap \s}^x \a = \int_{\g_x \cap \s}^{\Theta(x)}\; \|v\|_\mathsf g\, d\mathsf g$, and the similarly defined transformation $\Theta^{-1}$ are smooth in $\hat X$. 
\end{proof}

For $i= 1, 2$, let $\{\psi_i^t: SM_i \to SM_i\}_t$ denote the geodesic flow transformations, partially-defined for appropriate moments $t \in  \R$.

\begin{theorem}\label{th9.10}{\bf (the  strong topological rigidity of the geodesic flow for the inverse scattering problem in the presence of length data)}\smallskip

Let $(M_1, g_1)$ and  $(M_2, g_2)$ be two smooth compact connected Riemannian $n$-manifolds with boundaries.  Let the metric $g_1$ be boundary generic, and $g_2$ be of the gradient type, boundary generic, and balanced. 
 
Assume that the scattering maps $$C_{v^{g_1}}: \d_1^+(SM_1) \to \d_1^-(SM_1)\;\; \text{and} \;\; C_{v^{g_2}}: \d_1^+(SM_2) \to \d_1^-(SM_2)$$ are conjugated by a smooth diffeomorphism $\Phi^\d: \d_1(SM_1) \to \d_1(SM_2)$.   
Moreover, assume that for each $w \in \d_1^+(SM_1)$, the length data agree: $l_{g_2}(\Phi^\d(w)) = l_{g_1}(w)$.

Then 
\begin{itemize}
\item $g_1$ is a balanced metric of the gradient type.\smallskip

\item $\Phi^\d$ extends to a homeomorphism $\Phi: SM_1 \to SM_2$ such that $$\Phi \circ \psi_1^t(w) = \psi_2^t \circ \Phi (w)$$ for each $w \in SM_1$ and all $t \in \R$ for wich  $\psi_1^t(w)$ is well-defined.\smallskip

\item the restriction of $\Phi$ to each $v^{g_1}$-trajectory is a smooth diffeomorphism.\smallskip

\item If $g_2$ is such that no geodesic curve $\g \subset M_2$ is cubically tangent to $\d M_2$ at a pair of distinct points, and no geodesic curve is a singleton of multiplicity $4$, 
then the conjugating homeomorphism $\Phi: SM_1 \to SM_2$ is a smooth diffeomorphism. 
\end{itemize}  
\end{theorem}

\begin{proof} The proof is a modification of the arguments in the Holography Theorem 3.1 from \cite{K4}. As in that paper, we start with a balanced function $F_2:  SM_2 \to \R$ such that $dF_2(v^{g_2}) > 0$. We consider the pull-back $$F_1^\d =_{\mathsf{def}} (\Phi^\d)^\ast(F_2|_{\d X_2}): \d(SM_1) \to \R,$$ and using Lemma 3.2 from \cite{K4}, extend $F_1^\d$ to a smooth function $F_1: SM_1 \to \R$ so that $d F_1(v^{g_1}) > 0$. Then, as in the proof of the Holography Theorem, we define the scattering maps conjugating homeomorphism $\Psi: SM_1 \to SM_2$ so that: $\Psi|_{\d(SM_1)} = \Phi^\d$, $\Psi$ maps each $v^{g_1}$-trajectory $\tilde\g$ to a $v^{g_2}$-trajectory, the restriction of $\Psi|_{\tilde\g}$ is a diffeomorphism, and, due to  the construction of $\Psi$,  $F_1 = \Psi^\ast(F_2)$. 

By the latter property, for any $w \in \d_1^+(SM_1)$, we get the equality
$$\int_w^{C_{v^{g_1}}(w)} dF_1 =  \int_{\Psi(w)}^{C_{v^{g_2}}(\Psi(w))} dF_2.$$

Let $\mathsf g_i$, $i = 1, 2$, denote the Sasaki metric on $SM_i$, induced by the metric $g_i$ on $M_i$. Let $\pi_i: SM_i \to M_i$ be the obvious map. Since $v^{g_i}$ is orthogonal in $\mathsf g_i$ to the fibers $\pi_i^{-1}(\ast)$, we conclude that $\|v^{g_i}\|_{\mathsf g_i} = \|D\pi_i(v^{g_i})\|_{g_i} = 1$.  
\smallskip

Since $F_2$ is balanced, we get 
$$\int_{\Psi(w)}^{C_{v^{g_2}}(\Psi(w))} dF_2 = \int_{\Psi(w)}^{C_{v^{g_2}}(\Psi(w))} \|v^{g_2}\|_{\mathsf{g}_2} \,d\mathsf{g}_2 = \int_{\Psi(w)}^{C_{v^{g_2}}(\Psi(w))}  \,d\mathsf{g}_2,$$
the length of the geodesic arc $\g_{\Psi(w)}$ in $M_2$. 

On the other hand, by the hypotheses of the theorem, $$l_{g_1}(w) \;=_{\mathsf{def}} \; \int_{w}^{C_{v^{g_1}}(w)} \,d\mathsf{g}_1 = \int_{\Psi(w)}^{C_{v^{g_2}}(\Psi(w))} \,d\mathsf{g}_2 \;=_{\mathsf{def}}\; l_{g_2}(\Psi(w)).$$

So, $F_1$ is balanced as well. Therefore, Lemma \ref{lem9.x} is applicable to both pairs $(dF_1, \mathsf{g}_1)$ and $(dF_2, \mathsf{g}_2)$. By the lemma, there exist homeomorphisms $\Theta_1: SM_1 \to SM_1$ and $\Theta_2: SM_2 \to SM_2$ with the properties as in (\ref{eq9.?}).\smallskip

Finally, we construct a homeomorphism  $\Phi =_{\mathsf{def}} (\Theta_2)^{-1} \circ \Psi \circ \Theta_1$. For such a choice, thanks to properties (\ref{eq9.?}), we get $d\mathsf{g}_1|_\g = \Phi^\ast(d\mathsf{g}_2|_{\Phi(\g)})$ for all $v^{g_1}$-trajectories $\g$. So $\Phi \circ \psi_1^t(w) = \psi_2^t \circ \Phi (w)$ for each $w \in SM_1$ and all $t \in \R$ for which  $\psi_1^t(w)$ is well-defined. \smallskip

If $g_2$ is such that no geodesic curve $\g \subset M_2$ is cubically tangent to $\d M_2$ at a pair of distinct points and no geodesic curve is a singleton of multiplicity $4$, then by Lemma \ref{lem9.5}, $v^{g_2}$, and thus $v^{g_1}$ both possess property $\mathsf A$ from Definition \ref{def9.4a}. So, by Lemma \ref{lem9.x}, the conjugating homeomorphism $\Phi: SM_1 \to SM_2$ is a smooth diffeomorphism.
\end{proof}

Theorem \ref{th9.10} leads to the following ``\emph{Cut $\&$ Scatter Theorem}": 

\begin{theorem}\label{th9.11} Let $(N_1, g_1)$ and $(N_2, g_2)$  be two closed smooth Riemannian $n$-manifolds.  For $i = 1, 2$, let $U_i $ be a codimension zero submanifold of $N_i$ with a smooth boundary $\d U_i$.  Put $M_i =_{\mathsf{def}} N_i \setminus \textup{int}(U_i)$. Consider a compact neighborhood $V_i \subset N_i$ of $U_i$ whose interior contains $U_i$.

Assume that the metric $g_2|_{M_2}$ is boundary generic, of the gradient type, balanced, and satisfies property $\mathsf A$\footnote{If $U_2$ is strictly concave in $N_2$, then these hypotheses reduce to the requirement that $g_2|_{M_2}$ is of the gradient type and  balanced.}. 

Let a bijection $\psi: V_1 \to V_2$ be an \emph{isometry} with respect to $g_1|_{V_1}$ and $g_2|_{V_2}$. Consider the diffeomorphism $\Psi: SV_1 \to SV_2$, induced by $\psi$, and its restriction $\Psi^\d$ to $\d(SU_1)$.  \smallskip

If the scattering maps 
$$C_{v^{g_1|_{M_1}}}: \d_1^+(SM_1) \to \d_1^-(SM_1) \;\;  \text{and} \;\;\; C_{v^{g_2|_{M_2}}}: \d_1^+(SM_2) \to \d_1^-(SM_2)$$  
are conjugated  with the help of the diffeomorphism $\Psi^\d$ and if, for each $w \in \d_1^+(SM_1)$, the length data agree: $l_{g_2|_{M_2}}(\Psi^\d(w)) = l_{g_1|_{M_1}}(w)$, then the geodesic flows on $SN_1$ and $SN_2$ are \emph{strongly} smoothly conjugated in the sense of Definition \ref{def9.4}. 
\end{theorem} 

\begin{proof} Let $\tilde\g$ be a $v^{g_i}$-trajectory in $SN_i$, $i= 1, 2$. Put $\tilde\g_{M_i} =_{\mathsf{def}} \tilde\g \cap SM_i$,  $\tilde\g_{U_i} =_{\mathsf{def}} \tilde\g \cap SU_i$, and $\tilde\g_{V_i} =_{\mathsf{def}} \tilde\g \cap SV_i$. 
For each point $x \in \d(SU_1)$, denote by $\tilde\g(x)$ the $v^{g_1}$-trajectory through $x$ in $SN_1$.  

Since $\psi: V_1 \to V_2$ is an isometry, the diffeomorphism $\Psi$ maps $\tilde\g_{V_1}(x)$ to $\tilde\g_{V_2}(\Psi(x))$, while preserving the $(\mathsf g_i|_{V_i})$-induced natural parameterizations on the two cures. On the other hand, by Theorem \ref{th9.10}, the diffeomorphism $\Phi$ maps $\tilde\g_{M_1}(x)$ to $\tilde\g_{M_2}(\Psi(x)) = \Phi(\tilde\g_{M_1}(x))$, while preserving the $(\mathsf g_i|_{M_i})$-induced natural parameterizations on the two cures. 

Let us denote by $\Sigma_i$ the union of all \emph{segments} of the $v^{g_i}$-trajectories in $SM_i \cap SV_i$ that have a nonempty intersection with $\d(SM_i)$. Thus, any point of $\Sigma_i$ is connected to a point of $\d(SM_i)$ by an arc $\delta \subset \Sigma_i$ of a $v^{g_i}$-trajectory. Note that $\bar\Sigma_i$, the closure of $\Sigma_i$ in $SN_i$, is a compact set, containing $\d(SM_i)$. We denote by $\Sigma_i^\bullet$ the set $SU_i \cup \bar\Sigma_i$.

Since $\Psi |_{\d(SM_1)} = \Phi|_{\d(SM_1)}$, and both diffeomorphisms, $\Phi$ and $\Psi$, preserve the natural parameterizations along the trajectories in the source and the target, we conclude that $\Phi |_{\Sigma_1} = \Psi |_{\Sigma_1}$ as maps. It is possible that $\Psi$ and $\Phi$ may differ in $SV_1 \setminus \Sigma_1^\bullet$. \smallskip

Consider the homeomorphism $\Xi: SN_1 \to SN_2$, defined by the formula $$\Psi \bigcup_{\Psi^\d} \Phi: \; SU_1 \bigcup_{\d(SU_1)} SM_1 \to  SU_2 \bigcup_{\d(SU_2)} SM_2.$$

By the properties of $\Phi$ an $\Psi$, the homeomorphism $\Xi$ maps the $v^{g_1}$-trajectories in $SN_1$ to the $v^{g_2}$-trajectories in $SN_2$.

By the definitions of the sets $\Sigma_i, \Sigma_i^\bullet$ and using that  $\Phi |_{\Sigma_1} = \Psi |_{\Sigma_1}$, we may interpret $\Xi$ as the homeomorphism

 $$\Psi \bigcup_{\Psi |_{\d(\Sigma_1^\bullet)}} \Phi: \; \Sigma_1^\bullet \bigcup_{\d(\Sigma_1^\bullet)} \big(SM_1 \setminus \mathsf{int}(\Sigma_1)\big)  \to  \Sigma_2^\bullet \bigcup_{\d(\Sigma_2^\bullet)} \big(SM_2 \setminus \mathsf{int}(\Sigma_2)\big).$$ 

Since $\Phi |_{\Sigma_1} = \Psi |_{\Sigma_1}$,  $\Phi$ and $\Psi$ are smooth diffeomorphisms, and $\Sigma_1 \supset \d(SU_1)$, we conclude that $\Xi$ is a smooth diffeomorphism in the vicinity of $\d(SU_1)$. 

By Theorem \ref{th9.10}, the diffeomorphism $\Phi = \Xi |_{SM_1}$ has the property $\Xi^\ast(d \mathsf g_2|_{\Xi(\g^\#)}) =  d \mathsf g_1|_{\g^\#}$ for every $(v^{g_i |_{M_i}})$-trajectory $\g^\# \subset SM_i$. By the hypotheses, $\psi: U_1 \to U_2$ is an isometry. Thus, for every $(v^{g_i |_{U_i}})$-trajectory $\g^\dagger \subset SU_i$, we have $\Xi^\ast(d \mathsf g_2|_{\Xi(\g^\dagger)}) =  d \mathsf g_1|_{\g^\dagger}$. \smallskip

Let $\tilde\g$ be an integral curve of the geodesic field $v^{g_1}$ on $SN_1$. The set $\tilde\g_{U_1} = \tilde\g \cap SU_1$  is the union of trajectories $\{\g^\dagger\}$ of the geodesic field $v^{g_1|_{U_1}}$, while the set $\tilde\g_{M_1} = \tilde\g \cap SM_1$ is the union of $v^{g_1|_{M_1}}$-trajectories $\{\g^\#\}$. Therefore, by the arguments above,  $\Xi^\ast(d \mathsf g_2|_{\Xi(\tilde\g)}) =  d \mathsf g_1|_{\tilde\g}$ for all $v^{g_1}$-trajectories $\tilde\g$ in $SN_1$. This implies that $\Xi \circ \phi_1^t = \phi_2^t \circ \Xi$ for all moments $t \in \R$, where $\phi_i^t$ denotes the geodesic flow diffeomorphism of $SN_i$. \smallskip

Thus the metrics $g_1$ and $g_2$ are geodesic flow strongly conjugate by a diffeomorphism $\Xi: SN_1 \to SN_2$ from the class $C^\infty(SN_1, SN_2)$.
\end{proof}
 
In the next three corollaries, we combine Theorem \ref{th9.11} with a number of classical results that make it possible to reconstruct the metric on special closed manifolds from the corresponding geodesic flows.

\begin{corollary}\label{cor9.13}  Let the pairs of smooth Riemannian  $n$-manifolds, $(N_1, g_1) \supset (M_1, g_1|)$ and $(N_2, g_2)  \supset (M_2, g_2|)$, be as in Theorem \ref{th9.11}.

If the scattering maps $C_{v^{g_1|_{M_1}}}$ and $C_{v^{g_2|_{M_2}}}$  are conjugated  with the help of the diffeomorphism $\Psi^\d =_{\mathsf{def}} \Psi |_{\d(SU_1)}$\footnote{As in Theorem \ref{th9.11}, $\Psi$ is induced by an isometry $\psi: V_1 \to V_2$.} and if, for each $w \in \d_1^+(SM_1)$, the length data agree: $l_{g_2|_{M_2}}(\Psi^\d(w)) = l_{g_1|_{M_1}}(w)$, then the manifolds $(M_1, g_1|_{M_1})$ and $(M_2, g_2|_{M_2})$ share the same volume.
\end{corollary}
\begin{proof} By Theorem \ref{th9.11},  the metrics $g_1$ and $g_2$ are geodesic flow strongly conjugate by a diffeomorphism $\Xi: SN_1 \to SN_2$ from the class $C^\infty(SN_1, SN_2)$. By \cite{CK}, Proposition 1.2, the manifolds $(N_1, g_1)$ and $(N_2, g_2)$ share the same volume. Since $\psi: U_1 \to U_2$ is an isometry, the manifolds $(M_1, g_1|_{M_1})$ and $(M_2, g_2|_{M_2})$ share the same volume as well.
\end{proof}
\begin{corollary}\label{cor9.14}  Let the pairs of smooth Riemannian  $n$-manifolds, $(N_1, g_1) \supset (M_1, g_1|)$ and $(N_2, g_2)   \supset (M_2, g_2|)$, be as in Theorem \ref{th9.11}. In addition, assume that $N_2$ admits a non-vanishing \emph{parallel} vector field $w$ \footnote{The existence of a parallel vector field is equivalent to $(N_2, g_2)$ being isometric to the quotient of the Riemannian product $L \times \R$ by the isometry subgroup $G \subset \mathsf{Iso}(L) \times \mathsf{Iso}_+(\R)$, where $L$ is a simply connected complete Riemannian manifold (\cite{CK}).}.

If the scattering maps $C_{v^{g_1|_{M_1}}}$ and $C_{v^{g_2|_{M_2}}}$  are conjugated  with the help of the diffeomorphism $\Psi^\d =_{\mathsf{def}} \Psi |_{\d(SU_1)}$ and if, for each $w \in \d_1^+(SM_1)$, the length data agree: $l_{g_2|_{M_2}}(\Psi^\d(w)) = l_{g_1|_{M_1}}(w)$, then the manifolds $(N_1, g_1)$ and $(N_2, g_2)$ are isometric.
\end{corollary}
\begin{proof} By Theorem \ref{th9.11},  the metrics $g_1$ and $g_2$ are geodesic flow strongly conjugate by a diffeomorphism $\Xi: SN_1 \to SN_2$ from the class $C^\infty(SN_1, SN_2)$. By \cite{CK}, Theorem 1.1, the manifolds $(N_1, g_1)$ and $(N_2, g_2)$ are isometric. 
\end{proof}

\begin{corollary}\label{cor9.15} Let $n \geq 3$. Let the pairs of Riemannian  $n$-manifolds, $(N_1, g_1) \supset (M_1, g_1|)$ and $(N_2, g_2) \supset (M_2, g_2|)$, be as in Theorem \ref{th9.11}. Assume, in addition, that $(N_2, g_2)$ is locally symmetric and of negative sectional curvature. 

If the scattering maps $C_{v^{g_1|_{M_1}}}$ and $C_{v^{g_2|_{M_2}}}$  are conjugated  with the help of the diffeomorphism $\Psi^\d =_{\mathsf{def}} \Psi |_{\d(SU_1)}$ and if, for each $w \in \d_1^+(SM_1)$, the length data agree: $l_{g_2|_{M_2}}(\Psi^\d(w)) = l_{g_1|_{M_1}}(w)$, then the spaces $(N_1, g_1)$ and $(N_2, g_2)$ are isometric. 
\end{corollary}

\begin{proof} By Theorem \ref{th9.11}, the $C^1$-diffeomorphism  $\Xi: SN_1 \to SN_2$ conjugates the geodesic flows on $SN_1$ and $SN_2$. Since $(N_2, g_2)$ is a locally symmetric space of negative sectional curvature, the Besson-Courtois-Gallaot Theorem \ref{th9.2} applies; so we get that, for an appropriate  constant $c > 0$,  there exists an isometry $\chi: (N_1, c\cdot g_1) \to (N_2, g_2)$.  
Using that $\psi: U_1 \to U_2$ is an isometry, we conclude that $c = 1$.  
\end{proof}

\bigskip
{\it Acknowledgments:} The author is thankful to Christopher Croke and Gunther Uhlmann for very stimulating, informative conversations. He is also grateful  to Reed Meyerson for helping with the proof of Lemma 2.3.

\end{document}